\documentclass [leqno, 10pt]{amsart}
\usepackage[T1]{fontenc}
\usepackage[utf8]{inputenc}
\usepackage{amssymb, amsmath, amsfonts, amsthm, amsbsy, amscd}
\usepackage[bbgreekl]{mathbbol} 
\usepackage[mathscr]{euscript}
\usepackage{esint}
\usepackage{tabularx, caption}

\usepackage{url}
\usepackage[linktocpage]{hyperref}

\usepackage{filecontents}

\usepackage{etoolbox}
\apptocmd{\sloppy}{\hbadness 10000\relax}{}{}

\newtheorem{theorem}{Theorem}[section]
\newtheorem{corollary}[theorem]{Corollary}
\newtheorem{lemma}[theorem]{Lemma}

\numberwithin{equation}{section}

\newcommand{\s}{\mathfrak{s}}
\newcommand{\hdet}{\operatorname{\mathsf{Det}}}
\newcommand{\muf}{\mu}
\newcommand{\amc}{\mathcal{H}}

\newcommand{\cn}{\nu}
\newcommand{\gn}{N}

\newcommand{\gf}{F^{\sharp}}

\newcommand{\Om}{\Omega}

\newcommand{\amg}{\mathscr{A}}

\newcommand{\Grea}{G_{\rea}}
\newcommand{\stw}{\mathbb{W}}
\newcommand{\ster}{\ste_{\rea}}

\newcommand{\sing}{\mathbb{S}}
\newcommand{\pv}{(G, \ste, \rho)}
\newcommand{\pvr}{(G_{\rea}, \ste_{\rea}, \rho)}
\newcommand{\imt}{\iota}

\newcommand{\A}{\mathsf{A}}
\newcommand{\U}{\mathsf{U}}
\newcommand{\kc}{\mathsf{c}}

\newcommand{\quat}{\mathbb{H}}

\newcommand{\lc}{\Sigma}

\renewcommand{\H}{\mathsf{H}}
\newcommand{\Hc}{\mathsf{H}^{\com}}

\newcommand{\pol}{\mathsf{Pol}}

\newcommand{\om}{\omega}

\newcommand{\hess}{\operatorname{\mathsf{Hess}}}

\newcommand{\vol}{\mathsf{vol}}

\newcommand{\ka}{\kappa}

\renewcommand{\part}{\vdash}

\newcommand{\rad}{\mathbb{E}}

\newcommand{\Ga}{\Gamma}

\newcommand{\nm}{\mathsf{W}}

\newcommand{\lap}{\Delta}

\renewcommand{\j}{\mathsf{i}}

\newcommand{\la}{\lambda}
\newcommand{\ep}{\epsilon}

\newcommand{\reat}{\mathbb{R}^{\times}}
\newcommand{\comt}{\mathbb{C}^{\times}}

\newcommand{\ext}{\Omega}
\newcommand{\cinf}{C^{\infty}}

\newcommand{\eno}{\text{End}}

\newcommand{\si}{\sigma}
\newcommand{\pr}{\partial}
\newcommand{\bpr}{\bar{\partial}}

\newcommand{\sign}{\operatorname{sgn}}
\newcommand{\bnabla}{\bar{\nabla}}
\newcommand{\integer}{\mathbb{Z}}
\newcommand{\en}{[\nabla]}
\newcommand{\lie}{\mathfrak{L}}

\newcommand{\Aff}{\mathbb{Aff}}
\newcommand{\aff}{\mathbb{aff}}

\newcommand{\lb}{\langle}
\newcommand{\ra}{\rangle}

\newcommand{\ste}{\mathbb{V}}

\newcommand{\al}{\alpha}
\newcommand{\be}{\beta}
\newcommand{\ga}{\gamma}

\newcommand{\hnabla}{\widehat{\nabla}}
\newcommand{\tnabla}{\tilde{\nabla}}
\newcommand{\gl}{\mathfrak{gl}}

\newcommand{\eul}{\mathbb{X}}

\newcommand{\proj}{\mathbb{P}}
\newcommand{\pfaff}{\operatorname{Pf}}

\DeclareMathOperator{\Aut}{Aut}

\newcommand{\g}{\mathfrak{g}}

\newcommand{\ad}{\text{ad}}

\newcommand{\tensor}{\otimes}

\newcommand{\rea}{\mathbb R}
\newcommand{\com}{\mathbb C}

\begin{document}
\title{Functions dividing their Hessian determinants and affine spheres}
\author{Daniel J.~F. Fox} 
\address{Departamento de Matemáticas del Área Industrial\\ Escuela Técnica Superior de Ingeniería y Diseño Industrial\\ Universidad Politécnica de Madrid\\Ronda de Valencia 3\\ 28012 Madrid España}
\email{daniel.fox@upm.es}
\thanks{Acknowledgment: I thank Roland Hildebrand for comments on a preliminary version of this paper.}

\begin{abstract}
The nonzero level sets of a homogeneous, logarithmically homogeneous, or translationally homogeneous function are affine spheres if and only if the Hessian determinant of the function is a multiple of a power or an exponential of the function. In particular, the nonzero level sets of a homogeneous polynomial are proper affine spheres if some power of it equals a nonzero multiple of its Hessian determinant. The relative invariants of real forms of regular irreducible prehomogeneous vector spaces yield many such polynomials which are moreover irreducible. For example, the nonzero level sets of the Cayley hyperdeterminant are affine spheres. 
\end{abstract}

\maketitle

\section{Introduction}
Let $\hnabla$ be the standard flat affine connection on $\rea^{n+1}$ and fix a $\hnabla$-parallel volume form $\Psi$. Elements of the subgroup of the group $\Aff(n+1, \rea)$ of affine transformations of $\rea^{n+1}$ (comprising the automorphisms of $\hnabla$) preserving the tensor square $\Psi^{2}$ are called \textit{unimodular} or \textit{equiaffine}. For $F \in C^{k}(\rea^{n+1})$ let $F_{i_{1}\dots i_{k}} = \hnabla_{i_{1}}\dots\hnabla_{i_{k-1}}dF_{i_{k}}$, and let $g_{ij} = (\hess F)_{ij} = F_{ij} = \hnabla_{i}dF_{j}$ be the \textit{Hessian} of $F$. Here, as generally in what follows, the abstract index and summation conventions are employed. As $\det \hess F$ and $\Psi^{2}$ are $2$-densities, it makes sense to define the \textit{Hessian determinant} $\H(F)$ of a $C^{2}$ function $F$ by $\det \hess F = \H(F)\Psi^{2}$. In coordinates $x^{1}, \dots, x^{n}$ such that the coframe $dx^{1}, \dots, dx^{n+1}$ is $\hnabla$-parallel and $\Psi = dx^{1}\wedge \dots \wedge dx^{n+1}$, $\H(F) = \det \tfrac{\pr^{2}F}{\pr x^{i}\pr x^{j}}$. The adjugate tensor $U^{ij}$ of $F_{ij}$ is the symmetric bivector satisfying $U^{ip}F_{pj} = \H(F)\delta_{j}\,^{i}$. Where $\H(F) \neq 0$, $g_{ij}$ is a pseudo-Riemannian metric with inverse symmetric bivector $g^{ij}  = \H(F)^{-1}U^{ij}$. Define
\begin{align}\label{ufdet}
\U(F) = U^{ij}F_{i}F_{j} =  \H(F)|dF|^{2}_{g}  =  \begin{vmatrix} F_{ij} & 0 \\ F_{j}&  |dF|_{g}^{2}\end{vmatrix} = -\begin{vmatrix} F_{ij} & F_{i} \\ F_{j} & 0 \end{vmatrix}, 
\end{align}
where $|dF|_{g}^{2} = g^{ij}F_{i}F_{j}$ is not necessarily positive, since $g_{ij}$ is not necessarily positive definite. 
Consider the problem of finding a smooth function $F$ on an open domain of $\rea^{n+1}$ such that one or both of $\H(F)$ or $\U(F)$ is constant along the level sets of $F$, or, what is almost the same, there is a function $\phi(r)$ or $\psi(r)$ defined for $r$ in some open connected subset of $\rea$ such that
\begin{align}\label{mai}
&\H(F) = \phi(F),& &\text{or}& 
&\U(F) = \varphi(F).
\end{align}
Interesting choices for $\phi$ and $\varphi$ include polynomials, rational powers, and exponentials.

For $F \in \cinf(\rea^{n+1})$ and $g \in \Aff(n+1, \rea)$ define $(g \cdot F)(x) = F(g^{-1}x)$. Let $\ell:\Aff(n+1, \rea) \to GL(n+1, \rea)$ be the projection onto the linear part. Because of the identities 
\begin{align}\label{hgf}
&g\cdot \H(F) = \det{}^{2}\ell(g) \H(g\cdot F),& &g\cdot \U(F) = \det{}^{2} \ell(g) \U(g\cdot F),
\end{align}
(the second follows from \ref{ufdet}), the equations \eqref{mai} are affinely covariant in the sense that $F$ solves \eqref{mai} for some $\phi$ or $\varphi$ if and only if its precomposition with an affine transformation solves the same equation for some positive constant multiple of $\phi$ or $\varphi$. In particular, it is natural to consider solutions of \eqref{mai} up to unimodular affine equivalence. Moreover, the affine covariance suggests also that properties of the equations \eqref{mai} should be reflected in the unimodular affine geometry of the level sets of $F$. This suggests that when $F$ is replaced by $\psi \circ F$ for some $C^{2}$ function $\psi$ on some interval $I \subset \rea$ intersecting the image of $F$, then $\psi \circ F$ should also solve equations of the form \eqref{mai}, as the level sets of $\psi \circ F$ are the same as those of $F$, only differently parameterized. While because of the identities
\begin{align}\label{hpsif}
&\H(\psi \circ F) = \dot{\psi}^{n+1}(1 + (\ddot{\psi}/\dot{\psi})|dF|^{2}_{g})\H(F),&& 
\U(\psi \circ F) = \dot{\psi}^{n+2}\U(F),
\end{align}
in complete generality this need not be the case, if the form of either $\psi$ or $F$ is restricted, then it will sometimes be so, e.g. if $F$ has some kind of homogeneity. 

It is an interesting general problem to determine up to affine equivalence all sufficiently smooth solutions of \eqref{mai} on a domain $\Om \subset \rea^{n+1}$ for particular choices of $\phi$ or $\varphi$, e.g. when $\phi$ or $\varphi$ is a power or an exponential, and for particular choices of $\Om$. Of particular interest are solutions which are \textit{entire}, meaning defined on all of $\rea^{n+1}$, and domains which are \textit{cones}, meaning $e^{t}x \in\Om$ whenever $x \in \Om$ and $t \in \rea$. One aim in what follows is to identify some special cases of \eqref{mai} admitting nice solutions that can serve as candidates for answers to such characterization questions. A particularly interesting case of \eqref{mai} is the equation expressing that the Hessian determinant of a rational or polynomial function equal a multiple of a power of the function. The relative invariants of certain prehomogeneous vector spaces give many examples, as is explained later in the introduction.

The first result explicitly linking the equiaffine geometry of the level sets of $F$ with \eqref{mai} is Theorem \ref{ahtheorem} below that shows that any solution of \eqref{mai} with $\phi$ or $\psi$ a constant multiple of a power or an exponential and $F$ having appropriate homogeneity yields a one-parameter family of affine spheres. The precise statement requires some notation and terminology. An \textit{affine dilation} is an affine transformation mapping every line into a parallel line. An affine transformation is a dilation if and only if its linear part is a nonzero multiple of the identity, in which case it is a composition of a central homothety and a translation (which is regarded as a homothety with center at infinity). A one-parameter subgroup comprising affine dilations with a fixed center is either a one-parameter family $x \to e^{\la t}(x - v) + v$ of central homotheties with center $v$, or a one-parameter family $x \to x + \al t w$ of translations. A function $F \in C^{0}(\Om)$ is \textit{affinely homogeneous} on $\Om$ if there are one-parameter subgroups $\phi_{t} \in \Aff(n+1, \rea)$ and $\psi_{t} \in \Aff(1, \rea)$ of affine dilations with fixed center such that $F \circ \phi_{t}(x) = \psi_{t}\circ F(x)$ whenever $x$ and $\phi_{t}(x)$ are in $\Om$. Necessarily $\psi_{t}$ has the form $\psi_{t}(r) = e^{\la t}(r - r_{0}) + r_{0} + \al t$ where $\la \al = 0$; that is, at least one of $\la$ and $\al$ is $0$. Affine homogeneity is an affinely invariant condition in the sense that $F$ is affinely homogeneous if and only if $g\cdot F$ is affinely homogeneous for all $g \in \Aff(n+1, \rea)$. There are essentially four kinds of affinely homogeneous functions depending on whether each of $\phi_{t}$ and $\psi_{t}$ comprises central homotheties or translations. When both $\phi_{t}$ and $\psi_{t}$ comprise translations, the affine homogeneity simply means that the restriction of $F$ to any line with a given direction is an affine map. In this case the Hessian of $F$ is degenerate in the given direction, so this case requires a different treatment.

Let $\Om \subset \rea^{n+1}$ be an open subset, and for $F \in C^{0}(\Om)$, let $\lc_{r}(F, \Om) = \{x \in \Om: F(x) = r\}$. For $\la, \al \in \rea$ and $\ep \in \{0, 1\}$ define $\amg^{\la, \ep}_{\al}(\Om)$ to comprise those $F \in C^{0}(\Om) \cap \cinf(\Om \setminus \lc_{-\al/\la}(F, \Om))$ for which there exists a $v \in \rea^{n+1}$ such that $F((1-\ep) (e^{t}(x- v)+ v) + \ep (x + t v)) = e^{\la t}F(x) + \al t$ for all $t \in \rea$ and $x \in \Om$ such that $(1-\ep)( e^{t}(x - v) + v) + \ep (x + t v) \in \Om$. When $\la = 0$, $\lc_{-\al/\la}(F, \Om)$ is by convention empty.  Sometimes there will be written simply $F \in \amg^{\la, \ep}_{\al}$; in this case it is to be understood that $F$ is a $\cinf$ smooth function having the indicated homogeneity property on some open domain in $\rea^{n+1}$. Since $F \in \amg^{\la, 0}_{\al}$ satisfies $F(x) = F(e^{-t}e^{t}(x- v) + v) = F(x) + \al (e^{-\la t} - 1)t$, the condition defining $\amg^{\la, 0}_{\al}$ is vacuous unless $\al \la = 0$, that is, unless at least one of $\al$ and $\la$ is $0$. An element of $\amg^{\la, 0}_{0}$ is \textit{positively homogeneous} of \textit{degree} $\la$ (with center $v$), an element of $\amg^{0, 0}_{\al}$ is \textit{$\al$-logarithmically homogeneous} (with center $v$), and an element of $\amg^{\la, 1}_{0}$ is \textit{$\la$-translationally homogeneous} (with \textit{axial} center $v$). As remarked above, the Hessian of an element of $\amg^{0, 1}_{\al}$ is necessarily degenerate. However, since the exponential of an element of $\amg^{0, 1}_{\al}$ is an element of $\amg^{\al, 1}_{0}$, there is no need to consider the class $\amg^{0, 1}_{\al}$ as such. For the centrally homogeneous cases, it is often convenient to regard the center $v$ of the dilations as the origin $0$ in $\rea^{n+1}$. If $0$ is fixed initially, this can always be arranged by precomposing $F$ with translation by $-v$. 

Let $\Sigma$ be a co-orientable immersed hypersurface in $\rea^{n+1}$. Via the splitting $T\rea^{n+1} = T\Sigma \oplus \lb N\ra$ determined by a vector field $N$ transverse to $\Sigma$, the connection $\hnabla$ induces on $\Sigma$ a connection $\nabla$, a symmetric covariant two tensor $h$ representing the second fundamental form, a shape operator $S \in \Ga(\eno(T\Sigma))$, and the connection one-form $\tau \in \Ga(T^{\ast}\Sigma)$; these are defined by $\hnabla_{X}Y = \nabla_{X}Y + h(X, Y)N$ and $\hnabla_{X}N = -S(X) + \tau(X)N$, where $X$ and $Y$ are tangent to $\Sigma$. Tensors on $\Sigma$ are labeled using capital Latin abstract indices. That $\Sigma$ be \textit{nondegenerate} means that the second fundamental form of $\Sigma$, equivalently $h_{IJ}$, is nondegenerate. Since by assumption $\Sigma$ is co-oriented, it is orientable, and the interior multiplication $\imt(N)\Psi$ is a volume form on $\Sigma$. Since $\hnabla \Psi = 0$, for $X$ tangent to $\Sigma$, $\hnabla_{X} \imt(N)\Psi = \tau(X)\imt(N)\Psi$. Let $\vol_{h} = q\imt(N)\Psi$ be the volume form induced on $\Sigma$ by $h$ and the orientation consistent with $\imt(N)\Psi$. Since $\vol_{h}^{2} = |\det h|$,
\begin{equation}\label{deth}
h^{PQ}\nabla_{I}h_{PQ} = 2\vol_{h}^{-1}\nabla_{I}\vol_{h}  = 2\left(q^{-1}dq_{I} + \tau_{I} \right).
\end{equation}
Any other transversal to $\Sigma$ has the form $\tilde{N} = a(N + Z)$ for a nowhere vanishing function $a$ and a vector field $Z$ tangent to $\Sigma$. The second fundamental form $\tilde{h}$, connection $\tnabla$, and connection one-form $\tilde{\tau}$ determined by $\tilde{N}$ and $\hnabla$ are related to $h$, $\nabla$, and $\tau$ by
\begin{align}\label{transform}
&\tilde{h}_{IJ} = a^{-1}h_{IJ},& &\tnabla = \nabla - h_{IJ}Z^{K}, & &\tilde{\tau}_{I} = \tau_{I} + a^{-1}da_{I} + h_{IP}Z^{P}.
\end{align}
It follows from \eqref{deth} and \eqref{transform} that 
\begin{equation}\label{normalize}
n\tilde{\tau}_{I} + \tilde{h}^{PQ}\tnabla_{I}\tilde{h}_{PQ} = n \tau_{I} + h^{PQ}\nabla_{I}h_{PQ} + (n+2)Z^{P}h_{IP}, 
\end{equation}
where $h^{IJ}$ and $\tilde{h}^{IJ}$ are the symmetric bivectors inverse to $h_{IJ}$ and $\tilde{h}_{IJ}$. Since \eqref{normalize} does not depend on $a$, the span of $\tilde{N}$ is determined by requiring $n\tilde{\tau}_{I} =- \tilde{h}^{PQ}\tnabla_{I}\tilde{h}_{PQ}$, so that, by \eqref{deth} and \eqref{normalize},
\begin{equation}\label{zdet}
Z^{P}h_{PI} =  -\tfrac{1}{n+2}\left(n\tau_{I} + h^{PQ}\nabla_{I}h_{PQ}\right) = -\tau_{I} - \tfrac{2}{n+2}q^{-1}dq_{I}= -\tfrac{1}{2}h^{PQ}\nabla_{I}h_{PQ} + \tfrac{1}{n+2}q^{-1}dq_{I}.
\end{equation}
Whatever is $a$, the resulting transversal $\tilde{N}$ is called an \textit{affine normal}, and the line field it spans is called the \textit{affine normal distribution of $\Sigma$}. Since $\det \tilde{h} = a^{-n}\det h$, the \textit{equiaffine normal} $\nm = a(N + Z)$ is determined up to sign by requiring $|\vol_{\tilde{h}}| = |\imt(\nm)\Psi|$. By \eqref{deth}, $(n+2)\tilde{\tau}_{I} = 2\tilde{\tau}_{I} - \tilde{h}^{PQ}\tnabla_{I}\tilde{h}_{PQ} = 0$, so for the equiaffine normal the associated connection one-form vanishes. 

Once a co-orientation has been fixed, the pseudo-Riemannian metric $h$ and endomorphism $S$ determined by the co-oriented equiaffine normal are called the \textit{equiaffine metric} and the \textit{equiaffine shape operator}, respectively, and the \textit{affine mean curvature} is $\amc = (1/n)S_{I}\,^{I}$. A co-orientable nondegenerate connected hypersurface $\Sigma$ is an \textit{affine sphere} if and only if the shape operator determined by any affine normal is a multiple of the identity; that is, $S_{I}\,^{J} = n\amc\delta_{I}\,^{J}$. If $\Sigma$ is an affine sphere, then the affine mean curvature is constant, as follows from the traced form $\nabla_{Q}S_{I}\,^{I} = \nabla_{I}S_{Q}\,^{I}$ of the Gauss-Codazzi equations $\nabla_{[I}S_{J]}\,^{K} = 0$ for the equiaffine normal. An affine sphere is \textit{proper} or \textit{improper} according to whether its affine mean curvature is nonzero or zero. Clearly $\Sigma$ is an improper affine sphere if and only if the equiaffine normals are parallel. In this case $\Sigma$ is said to have \textit{center at infinity}. Similarly, $\Sigma$ is a proper affine sphere if and only if the equiaffine normals meet in a point, the \textit{center}. Precisely, if $x \in \Sigma$ then this point is $v = \eul_{x} + \amc^{-1}\nm_{x}$, where $\eul$ is the generates the radial dilations around 0, so that $\nm  = -\amc\rad$ where $\rad = \eul - v$. 

Theorem \ref{ahtheorem} can be summarized imprecisely as saying that the level sets of an affinely homogeneous function $F$ with a given center $v$ are affine spheres with center $v$ if and only if $\H(F)$ is a function of $F$. Understood in the most liberal way the forward implication of this statement is not strictly correct, because the level sets of $F$ need not be connected and a priori it is not obvious what it means for a disconnected hypersurface to be an affine sphere.

The following simple example indicates the need for some care regarding connectedness. Let $G(x) = x_{1}^{2} - x_{2}^{2} - x_{3}^{2}$. While $F(x) = G(x)^{2}$ solves $\H(F)^{2} = 2^{12}3^{2}F^{3}$, there is no $\phi$ such that $\H(F) = \phi(F)$ on all of $\rea^{n+1}$. On the other hand, on the open domains $\Om_{\pm} = \{x\in \rea^{n+1}: \pm G(x) > 0\}$ the function $F$ solves $\H(F) = \pm 2^{6}3F^{3/2}$. Each nonzero level set of $F$ is a union of a two-sheeted hyperboloid and a one-sheeted hyperboloid. Although each connected component is a proper affine sphere, the equiaffine metric of the two-sheeted hyperboloid is definite, while that of the one-sheeted hyperboloid is split. While it is reasonable to regard the two-sheeted hyperboloid as a single disconnected affine sphere, it is not reasonable to regard the level set of $F$ containing them as a disconnected affine sphere, because its components are affine spheres of different types.

The definition of a connected affine sphere does not require a choice of co-orientation, but some coherence condition on co-orientations is necessary when there are multiple connected components. The convention employed throughout this paper is the following. A smoothly immersed hypersurface having more than one connected component is an affine sphere if 
\begin{enumerate}
\item each connected component is an affine sphere and the centers of the different components are all the same, meaning that either the affine normal lines all meet in a common point, or are all parallel; 
\item there is a choice of co-orientations of the components so that for the equiaffine normal consistent with this choice the affine mean curvatures and the signatures modulo $4$ of the equiaffine metrics of the different components are all the same.
\end{enumerate} 
The signature of a nondegenerate symmetric bilinear form means the number of positive eigenvalues minus the number of negative eigenvalues. Over $\rea$ two nondegenerate symmetric bilinear forms have the same signature modulo $4$ if and only if their determinants have the same sign. Notice that if a disconnected hypersurface is an affine sphere with respect to a given choice of co-orientations of the components, it is an affine sphere with respect to the opposite choice of co-orientations, but with respect to no other choice of co-orientations. In this sense, the definition is consistent with the definition for a connected hypersurface. With this definition the nonzero level sets of $G  = x_{1}^{2} - x_{2}^{2} - x_{3}^{2}$ are affine spheres, while the nonzero level sets of $F = G^{2}$ are not. The stipulation modulo $4$ in the definition is necessary. For example, the determinant $P(X)$ of a $3 \times 3$ symmetric matrix $X$ satisfies $\H(P) = -16P^{2}$ and the level set $P(X) = 1$ is a disconnected affine sphere having connected components with equiaffine metrics of signatures $5$ and $1$.

Write $\sign(r) = r|r|^{-1}$ for the sign character of the group $\reat$ of nonzero real numbers. Define the \textit{standard co-orientation} of a connected component of a level set of $F$ to be that consistent with the vector field $-\sign(|dF|^{2}_{g})F^{i} = -\sign(\U(F)\H(F))F^{i}$, where $F^{i} = g^{ip}F_{p}$. 

\begin{theorem}\label{ahtheorem}
Let $\ep \in \{0, 1\}$ and let $\al$ and $\la$ be real constants not both $0$ such that $\al \la = 0$, $\al \ep = 0$, and $\la \neq 1 - \ep$. Let $\Om \subset \rea^{n+1}$ be a nonempty open subset, and let $I \subset \rea$ be a nonempty, connected, open subset such that $\la r + \al \neq 0$ for all $r \in I$. For $F \in \amg^{\la, \ep}_{\al}(\Om)$, let $\Om_{I} = F^{-1}(I)\cap \Om$. Let $v \in \rea^{n+1}$ be the center or axial center of $F$ as $\ep = 0$ or $\ep = 1$. The conditions
\begin{enumerate}
\item \label{aht2} There is a nonvanishing function $\phi:I \to \rea$ such that $F$ solves $\H(F) = \phi(F)$ on $\Om_{I}$;
\item \label{aht3} There is a nonvanishing function $\psi:I \to \rea$ such that $F$ solves $\U(F) = \psi(F)$ on $\Om_{I}$;
\end{enumerate}
are equivalent. If $\ep = 0$ the conditions \eqref{aht2} and \eqref{aht3} are equivalent to the condition
\begin{enumerate}
\setcounter{enumi}{2}
\item \label{aht1proper} For all $r \in I$ each level set $\lc_{r}(F, \Om_{I})$, equipped with the co-orientation of its components consistent with $-\sign(\U(F)\H(F))F^{i}$, is a proper affine sphere with center $v$,
\end{enumerate}
while if $\ep = 1$ the conditions \eqref{aht2} and \eqref{aht3} are equivalent to the condition
\begin{enumerate}
\setcounter{enumi}{3}
\item \label{aht1improper} For all $r \in I$ each level set $\lc_{r}(F, \Om_{I})$, equipped with the co-orientation of its components consistent with $-\sign(\U(F)\H(F))F^{i}$, is an improper affine sphere with equiaffine normal equal to $cv$ for a constant $c$ depending only on $r$ (and not the connected component).
\end{enumerate}
When these conditions hold, there is a nonzero constant $B$ such that $\phi$ and $\psi$ have the forms 
\begin{align}\label{phiforms}
\begin{split}
\phi(r) &= B|r|^{(n+1)(\la - 2)/\la}\,\,  \text{and}\,\, \psi(r) = \tfrac{\la}{\la - 1}r\phi(r)\,\, \text{if}\,\,\al = 0,\, \ep =0 \\
\phi(r) &= Be^{-2(n+1)r/\al}\,\,  \,\quad\text{and}\,\, \psi(r) = -\al\phi(r)\,\, \,\,\,\,\,\text{if}\,\, \la = 0, \,\ep = 0,\\
\phi(r) &= Br^{n+1}\,\,  \text{and}\,\, \psi(r) = r\phi(r)\,\, \text{if}\,\, \al = 0, \,\ep = 1 \\
\end{split}
\end{align} 
and the affine mean curvature of $\lc_{r}(F, \Om_{I})$ is 
\begin{align}\label{ahmc}
\begin{split}
&\sign(\la r)|\la - 1|^{-1/(n+2)}|\la|^{-(n+1)/(n+2)}|B|^{1/(n+2)}|r|^{(\la - 2)/\la - (n+1)/(n+2)} \,\,\,  \text{if}\,\, \al = 0,\, \ep = 0;\\
&\sign(\al)|\al|^{-(n+1)/(n+2)}|B|^{1/(n+2)}e^{-2(n+1)r/\al(n+2)}\qquad\qquad\qquad\qquad\quad\,\,  \text{if}\,\,\la = 0, \,\ep = 0,
\end{split}
\end{align}
or $0$ if $\al = 0$ and $\ep = 1$.
\end{theorem}
The elementary proof of Theorem \ref{ahtheorem} is given at the end of section \ref{characterizationsection}. If both $\U(F)$ and $\H(F)$ (equivalently $\H(F)$ and $|dF|_{g}^{2}$) are constant along the level sets of $F$ then the external reparameterization $\psi \circ F$ will have the same properties with no further restrictions on the form of $F$ besides obvious regularity assumptions. Recall that a function on a Riemannian manifold is \textit{isoparametric} if its Laplacian and the squared-norm of its differential are both constant along its level sets (see e.g. \cite{Cecil} or \cite{Thorbergsson}). The equations \eqref{mai} have a form similar to the classical isoparametric condition, though with the Laplacian replaced by the Hessian, and the squared Euclidean norm of the differential replaced by $\U(F)$. Lemma \ref{autoisolemma} shows that an $F$ solving \eqref{mai} is in fact isoparametric with respect to the metric $g$. Lemmas \ref{uflemma} and \ref{spherechartlemma} show that there hold the two equations \eqref{mai} on some open subset on which $g$ is nondegenerate if and only if the vector field $\nm$ formed by affine normals to the level sets of $F$ has the properties that its flow preserves the level sets of $F$ and its integral curves are contained in straight lines. If $F$ is moreover homogeneous, then the equations \eqref{mai} are redundant, and the level sets of $F$ must be affine spheres. 

Although no regularity assumptions are made on the functions $\phi$ and $\psi$ in \eqref{aht2} and \eqref{aht3}, such conditions follow automatically from the supposed affine homogeneity of $F$. This hypothesis is stronger than it may appear, and it would be interesting to deduce the conclusion of Theorem \ref{ahtheorem} from milder hypotheses. Note that in the convex case results with a similar flavor characterizing ellipsoids and elliptic paraboloids have been obtained by D.-S. Kim and Y.~H. Kim in \cite{Kim-Kim-i} and \cite{Kim-Kim-ii}. 

An unsatisfying aspect of Theorem \ref{ahtheorem} is the assumption in \eqref{aht1proper} and \eqref{aht1improper} that the centers of the affine spheres coincide with the dilation center of $F$. While it is reasonable to assume that level sets have a common center, that it coincides with that of $F$ should be a conclusion of the theorem.

The example $G = x_{1}^{2} - x_{2}^{2} - x_{3}^{2}$ shows the importance of the choice of domain in Theorem \ref{ahtheorem}. For this $G$, conclusion \eqref{aht2} of the theorem is true on the domains $\Om_{\pm}$, but not on all of $\rea^{n+1}$, because the hypothesis \eqref{aht1proper} is satisfied on $\Om_{\pm} = \{x\in \rea^{n+1}: \pm G(x) > 0\}$, but not on all of $\rea^{n+1}$.

Theorem \ref{ahtheorem} raises the question of the existence and uniqueness of solutions to \eqref{mai} for $\phi$ and $\psi$ as in Theorem \ref{ahtheorem}. In a setting where $\H$ need not be elliptic, or $\Om$ need not be convex, it seems even precise questions are lacking. 
On the other hand, while much remains to be said in the nonelliptic/nonconvex case, in the elliptic/convex case there is an extensive general theory for \eqref{mai}, and the convex affine spheres are constructed in full generality due principally to work of S.~Y. Cheng and S.~T. Yau in \cite{Cheng-Yau-mongeampere, Cheng-Yau-realmongeampere, Cheng-Yau-affinehyperspheresI} (see also  \cite{Sasaki} and the surveys \cite{Fox-schwarz}, \cite{Loftin-survey}, and \cite{Trudinger-Wang-survey}). A cone is \textit{proper} if it is nonempty and its closure contains no complete affine line. The group $\Aut(\Om) \subset \Aff(n+1, \rea)$ of affine automorphisms of an open subset $\Om \subset \rea^{n+1}$ comprises those $g \in \Aff(n+1, \rea)$ such that $g\Om \subset \Om$. An open $\Om$ is \textit{affinely homogeneous} if $\Aut(\Om)$ acts transitively on $\Om$. A convex affine sphere is \textit{hyperbolic} if its affine mean curvature is negative.

\begin{theorem}[S.~Y. Cheng and S.~T. Yau]\label{cytheorem}
A proper open convex cone $\Om \subset \rea^{n+1}$ is foliated in a unique way by hyperbolic affine spheres asymptotic to its boundary. More precisely, there exists a unique smooth convex function $F:\Om \to \rea$ solving $\H(F) = e^{2F}$, tending to $+\infty$ on the boundary of $\Om$, such that $g_{ij} = \hnabla_{i}dF_{j}$ is a complete Riemannian metric on $\Om$, and such that the level sets of $F$ are properly embedded hyperbolic affine spheres foliating $\Om$. 
Moreover, $(g\cdot F)(x) = F(x) + \log \det \ell(g)$ for all $g \in \Aut(\Om)$, so that $g_{ij}$ is $\Aut(\Om)$ invariant and $F$ is $-(n+1)$-logarithmically homogeneous.
\end{theorem}
By Theorem \ref{cytheorem}, a geometric picture that holds for a homogeneous convex cone and its characteristic function holds for a general proper convex cone provided that the solution of $\H(F) = e^{2F}$ is used in place of the characteristic function. As will be explained next, together Theorem \ref{matheorem} and Theorem \ref{pvtheorem}, showing that interesting explicit polynomial solutions of \eqref{mai} arise as the relative invariants of certain prehomogeneous vector spaces, provide a supply of explicit solutions of \eqref{mai} in both the convex and nonconvex cases. In particular, the open orbits of real forms of certain prehomogeneous vector spaces are foliated by proper affine spheres arising as the level sets of a solution of a Monge-Amp\`ere equation, and this is the homogeneous geometric picture on which should be modeled an analogue of Theorem \ref{cytheorem}  applying to some class of nonconvex cones. 

Let $\pol^{k}(\rea^{n+1})$ denote the space of degree $k$ homogeneous polynomials on $\rea^{n+1}$. Let $k > 1$ be an integer that divides $(n+1)(k-2)$, or, what is the same, divides $2(n+1)$. For a sufficiently smooth function $F$ on $\rea^{n+1}$ consider the equation 
\begin{align}\label{ma}
&\H(F) = \kc F^{m},& &m = (n+1)(k-2)/k,
\end{align}
for a nonzero constant $\kc$. The exponent in \eqref{ma} is explained by the observation that since $\H(P) \in \pol^{(n+1)(k-2)}(\rea^{n+1})$, for $P$ to solve $\H(P) = \kc P^{m}$ with $\kc \neq 0$, it must be that $km = (n+1)(k-2)$. 

\begin{theorem}\label{matheorem}
If $P \in \pol^{k}(\rea^{n+1})$ solves \eqref{ma} for a nonzero constant $\kc$, then, for nonzero $r$, each connected component of the level set $\lc_{r} = \{x \in \rea^{n+1}: P(x) = r\}$ is a proper affine sphere. Conversely, if for each nonzero $r$ each level $\lc_{r}$ of a homogeneous polynomial $P \in \pol^{k}(\rea^{n+1})$ is an affine sphere, then $P$ solves \eqref{ma}. 
\end{theorem}

\begin{proof}
That the nonzero level sets of a polynomial solving \eqref{ma} are proper affine spheres is immediate from Theorem \ref{ahtheorem}. Suppose that $P \in \pol^{k}(\rea^{n+1})$ and for each nonzero $r$ each connected component of $\lc_{r}(P)$ is a proper affine sphere centered on the origin. By Lemma \ref{homognondegenlemma}, $\H(P)$, $\U(P)$, and $P^{i}$ do not vanish on $\lc_{r}(P)$. Let $\Sigma$ be a connected component of $\lc_{r}(P)$. Let $\Sigma_{t}$ be the image of $\Sigma$ under dilation by a factor of $e^{t}$ around the origin. It is a connected component of $\lc_{e^{kt}r}(P)$. By the tubular neighborhood theorem the union $\Om = \cup_{t \in \rea}\Sigma_{t}$ is an open cone on which the hypotheses of \eqref{aht2} of Theorem \ref{ahtheorem} are satisfied for the polynomial $P$, and so there is a constant $B \neq 0$ such that, on $\Om$, $\H(P) = B|P|^{m}$, where $m = (n+1)(k-2)/k$. Since $\H(P)$ is a polynomial, $m$ must be an integer, so that there is a nonzero constant $c$ (equal to $\pm B$) such that $\H(P) = cP^{m}$ on $\Om$. Since the homogeneous polynomial $\H(P) - \kc P^{m}$ vanishes on the open set $\Om$, it is identically zero.
\end{proof}

The conditions that $P$ be homogeneous, that $\kc$ be nonzero, and that the affine spheres be proper are necessary in Theorem \ref{matheorem}. The polynomial $P = x_{n+1} - \sum_{i = 1}^{n}x_{i}^{2}$ solves $\H(P) = 0$, and its nonzero level sets are paraboloids, which, although affine spheres, are improper. More generally:

\begin{lemma}\label{nonhomoglemma}
Let $P \in \pol^{2}(\rea^{n})$ be a nondegenerate quadratic form and let $\kc = \H(P) \neq 0$. If $l > 1$ is an integer dividing $n+2$ then $Q = (x_{n+1} - P(x_{1}, \dots, x_{n}))^{l}$ solves $\H(Q) = (-1)^{n}l^{n+1}(l-1)\kc Q^{n+1 - (n+2)/l}$. 
However, $Q$ is not affinely equivalent to a homogeneous polynomial.
\end{lemma}

\begin{proof}
That $Q$ solves the equation $\H(Q) = \ka Q^{n+1 - (n+2)/l}$ with $\ka =  (-1)^{n}l^{n+1}(l-1)\kc$ is a straightforward calculation. Were there $g\in \Aff(n+1, \rea)$ such that $g\cdot Q$ were homogeneous, then since, by \eqref{hgf}, $g\cdot Q$ would solve $\H(g\cdot Q) = \ka (\det \ell(g))^{-2}(g\cdot Q)^{n+1 - (n+2)/l}$, by Theorem \ref{matheorem}, the nonzero level sets of $g \cdot Q$ would be proper affine spheres. However, the nonzero level sets of $Q$ are contained in graphs of functions of the form $P + c$, with $c$ a constant, and these are improper affine spheres.
\end{proof}

By \eqref{hgf}, if $F$ solves \eqref{ma} for the constant $\kc$ then $g \cdot F$ solves \eqref{mai} for the constant $\kc \det^{-2}\ell(g)$, so while the value of $\kc$ has no affinely invariant significance, its sign has. In the simplest case of \eqref{ma}, when $m = 0$, an entire locally uniformly convex solution of \eqref{ma} with $\kc > 0$ is a quadratic polynomial (\cite{Jorgens, Calabi-improper, Pogorelov}). According to Remark $(ii)$ following Theorem $4.4$ of \cite{Trudinger-Wang-survey}, the local uniform convexity is not necessary for the conclusion. On the other hand, if $\kc < 0$, any function of the form $F(x_{1}, \dots, x_{n+1}) = (-\kc)^{1/2}x_{1}x_{n+1} + f(x_{1}) + \tfrac{1}{2}\sum_{i = 2}^{n}x_{i}^{2}$ with $f \in C^{2}(\rea)$ solves \eqref{ma} with $m = 0$ on all of $\rea^{n+1}$. For $m > 0$ the situation is more complicated in that there is an abundance of homogeneous polynomial solutions of \eqref{ma} having degree at least $3$. The homogeneous polynomials of two variables solving an equation of the form \eqref{ma} with $\kc \neq 0$ are distinguished up to affine equivalence by degree, signature of the Hessian, and the sign of $\kc$; all are products of powers of linear forms and quadratic polynomials. Similarly, a ternary cubic polynomial solving $\H(P) = \kc P$ for $\kc \neq 0$ is affinely equivalent to a product of linearly independent linear forms if $\kc > 0$, and to a product of a linear form and degenerate quadratic form if $\kc <0$. Among other things, this implies that a cubic ternary polynomial solution of \eqref{ma} is necessarily reducible and decomposable as a product of lower dimensional solutions. In general, products of appropriate powers of polynomial solutions of \eqref{ma} for some $m$ and $\kc$ are again solutions, for some other $m$ and $\kc$. In particular, a product of powers of linear forms and quadratic polynomials solves an equation of the form \eqref{ma}. While this shows that equations of the form \eqref{ma} admit solutions of arbitrarily many variables and arbitrary degrees, it also suggests that polynomial solutions of \eqref{ma} are most interesting when they are irreducible as polynomials. As is explained precisely in section \ref{prehomogeneoussection}, a large class of irreducible polynomials solving \eqref{ma} is given by the relative invariants of irreducible regular real prehomogeneous vector spaces, (note that \textit{irreducible} is used here with two different meanings).  

\begin{theorem}\label{pvtheorem}
\noindent
\begin{enumerate}
\item\label{pv1} Given a real form of a regular irreducible complex prehomogeneous vector space $\pv$, the restriction $P$ to the real points $\ste_{\rea}$ of $\ste$ of an appropriate scaling of the relative invariant of $\pv$ is an irreducible homogeneous real polynomial solving \eqref{ma}. A connected component of a nonzero level set of such a $P$ is a homogeneous affine sphere.
\item\label{pv2} If the complex polynomial $P$ is the relative invariant of a regular irreducible complex prehomogeneous vector space $\pv$ then the real polynomial $|P|^{2} = P\bar{P}$ solves \eqref{ma}. A connected component of a nonzero level set of $|P|^{2}$ is a homogeneous affine sphere.
\end{enumerate}
\end{theorem}
\noindent
An affine sphere is \textit{homogeneous} if it is contained in an orbit of some group of affine transformations.
The conclusion of Theorem \ref{pvtheorem} is valid under a hypothesis milder than the irreducibility of $\pv$, namely that the singular set of $\pv$ be an irreducible hypersurface. The result is proved in this generality as Theorem \ref{pv2theorem}. 

In conjunction with the Sato-Kimura classification of irreducible regular complex prehomogeneous vector spaces, Theorems \ref{matheorem} and \ref{pvtheorem} give many explicit examples of affine spheres and solutions of \eqref{ma}. Among the simplest examples are the nonzero level sets of the discriminant,
\begin{equation}\label{pvexample}
P = x_{2}^{2}x_{3}^{2} + 18x_{1}x_{2}x_{3}x_{4} - 4x_{1}x_{3}^{3} - 4x_{2}^{3}x_{4} - 27x_{1}^{2}x_{4}^{2},
\end{equation}
of the binary cubic form $f(u, v) = x_{1}u^{3} + x_{2}u^{2}v + x_{3}uv^{2} + x_{4}v^{3}$. 
This solves $\H(P) = 2^{4}3^{5}P^{2}$. The action of $GL(2, \rea)$ on $S^{3}(\rea^{2})$ is irreducible and prehomogeneous, and has $P$ as relative invariant, so by Theorem \ref{pvtheorem}, the nonzero level sets of $P$ are proper affine spheres homogeneous for the induced action of $SL(2, \rea)$. In this case the Hessian metric $g$ has split signature, and so the equiaffine metrics of the resulting three-dimensional affine spheres have indefinite signature. A more impressive example comes from the prehomogeneous action of the real form $GL(1, \rea) \times E_{6(6)}$ on the $3 \times 3$ Hermitian matrices over the split octonions. The relative invariant is the cubic form stabilized by $E_{6(6)}$ and can be identified with $P(X, u, v) = \pfaff(X) + u^{t}Xv$ where $X$ is a $6\times 6$ skew-symmetric matrix and $u, v \in \rea^{6}$. Then $P \in \pol^{3}(\rea^{27})$ solves $\H(P) = 2P^{9}$. For illustration, an incomplete list of examples obtained in this way is given in section \ref{prehomogeneoussection}; the preceding examples are entries \ref{cubicsplit} and \ref{e6ex2}. A perusal of the examples shows that most can be realized as level sets of determinants, Pfaffians, or discriminants (in the sense of \cite{Gelfand-Kapranov-Zelevinsky}). For example, the Cayley hyperdeterminant (\cite{Cayley-hyperdeterminants}) of a covariant $3$-tensor on a two-dimensional complex vector space $\ste$ arises as the relative invariant for the irreducible action of $GL(2, \com)\times GL(2,\com)\times GL(2,\com)$ on $\tensor^{3}\ste^{\ast}$. This action has seven orbits, so Theorem \ref{pvtheorem} applies. If $X \in \tensor^{3}\ste^{\ast}$ is written as the polynomial
\begin{equation}
X = x_{1}u^{3} + x_{2}u^{2}v + x_{3}uvu + x_{4}vu^{2} + x_{5}uv^{2} + x_{6}vuv + x_{7}v^{2}u + x_{8}v^{3},
\end{equation}
in noncommuting variables $u$ and $v$, then its hyperdeterminant is
\begin{align}\label{chdet}
\begin{split}
P(X) &= \hdet X 
 = (x_{1}^{2}x_{8}^{2} + x_{2}^{2}x_{7}^{2} + x_{3}^{2}x_{6}^{2} + x_{4}^{2}x_{5}^{2})  + 4(x_{1}x_{5}x_{6}x_{7} + x_{2}x_{3}x_{4}x_{8})\\&- 2(x_{1}x_{2}x_{7}x_{8} + x_{1}x_{3}x_{6}x_{8} + x_{1}x_{4}x_{5}x_{8} + x_{2}x_{3}x_{6}x_{7} + x_{2}x_{4}x_{5}x_{7} + x_{3}x_{4}x_{5}x_{6}).
\end{split}
\end{align} 
The polynomial \eqref{chdet} solves $\H(P) = 2^{8}3P^{4}$, and a nonzero level set of \eqref{chdet} in $\tensor^{3}\rea^{2}$ is a proper affine sphere, homogeneous for the induced action of $SL(2, \rea)\times SL(2, \rea) \times SL(2, \rea)$, and with induced metric of maximally mixed signature. More generally, Corollary \ref{hyperdeterminantcorollary} shows that if the action $\rho$ of $G = GL(k_{1}+1, \com)\times \dots \times GL(k_{r}+1, \com)$ on the dual $\stw$ of the outer tensor product of the standard representations of its factors is prehomogeneous, then the hyperdeterminant of format $(k_{1} + 1, \dots, k_{r}+1)$ is the fundamental relative invariant for this action, and hence its nonzero levels over the real field are proper affine spheres. In general the prehomogeneous vector space $(G, \rho, \stw)$ is not reduced, and Lemma \ref{reducedlemma} shows that in this case the hyperdeterminant of format $(k_{1} + 1, \dots, k_{r}+1)$ can be obtained via castling from a standard determinant of a square matrix, the hyperdeterminant of format $(2, 2, 2)$, or the hyperdeterminant of format $(3, 3, 2)$. 

By Theorem \ref{ahtheorem}, the level sets of a translationally homogeneous function $F$ satisfying $\H(F) = \kc F^{n+1}$ are improper affine spheres. The equation $\H(F) = \kc F^{n+1}$ results from \eqref{ma} when the homogeneity degree $k$ tends to $\infty$, so in some formal sense the analogue for this equation of degree $k$ homogeneous polynomial solutions of \eqref{ma} should be functions that somehow can be regarded as polynomials homogeneous of infinite degree. It turns out that this makes sense if one regards a translationally homogeneous exponential of a weighted homogeneous polynomial as having infinite homogeneity degree. Precisely, there can be constructed translationally homogeneous solutions $F$ of $\H(F) = \kc F^{n+1}$ having the form $F = e^{P}$ where the weighted homogeneous polynomial $P$ is the characteristic polynomial of a left symmetric algebra for which the right trace form is nondegenerate, the left multiplication operators are triangular in some basis, and the derived Lie subalgebra has codimension one. The level sets of $P$ are homogeneous improper affine spheres. The simply-connected solvable Lie group corresponding to the underlying Lie algebra acts with an open orbit equal to the complement of the zero set of its relative invariant $P$, and so these examples fit into a common framework with the homogeneous proper affine spheres described above. However, since recounting the background on left symmetric algebras necessary to prove these claims takes considerable space, these results will be reported elsewhere.

\section{Auto-isoparametric conditions and the proof of Theorem \ref{ahtheorem}}\label{characterizationsection}
Lemma \ref{affnormallemma} gives an explicit formula for the affine normal of a level set that coincides with that of Theorem $1$ of \cite{Hao-Shima}, modulo differences in notation. 

\begin{lemma}[J. Hao and H. Shima, \cite{Hao-Shima}]\label{affnormallemma}
Let $F$ be a $\cinf$ function defined on some open subset of $\rea^{n+1}$ and let $g_{ij} = \hnabla_{i}dF_{j}$. Let $\Om$ be a (Euclidean) connected component with nonempty interior of the region on which $F$, $dF$, $\H(F)$, and $\U(F)$ are nonvanishing. For $r \in \rea$ let $\lc_{r}(F, \Om) = \{x \in \Om: F(x) = r\}$. By assumption $\H(F)$ does not change sign on $\lc_{r}(F, \Om)$ and $F^{i} = g^{ij}F_{j}$ is nonzero on $\Om$, so is transverse to $\lc_{r}(F, \Om)$, which is therefore co-orientable. Let $\muf = (n+2)^{-1}d\log\U(F)$ and $\A(F) = 1 - F^{p}\muf_{p}$. The equiaffine normal of $\lc_{r}(F, \Om)$ consistent with the co-orientation determined by $-\sign(\H(F)\U(F))F^{i} = -\sign(|dF|^{2}_{g})F^{i}$ has the explicit expression 
\begin{align}\label{affnorm}
\begin{split}
\nm^{i} &= -|\U(F)|^{1/(n+2)}\left(\U(F)^{-1}\H(F)\A(F)F^{i} + \muf^{i}\right)  = \tilde{g}^{ip}\left(\A(F)|\cn|_{\tilde{g}}^{-2}\cn_{p} - \muf_{p}\right),
\end{split}
\end{align} 
where the  the \textit{equiaffine conormal} one-form $\cn_{i} = -|\U(F)|^{-1/(n+2)}F_{i}$ annihilates the tangent space to $\lc_{r}(F, \Om)$ and satisfies $\cn_{i}\nm^{i} = 1$, the equiaffine metric of $\lc_{r}(F, \Om)$ equals the restriction of 
\begin{align}\label{hnablacn}
\begin{split}
\tilde{g}_{ij}  = |\U(F)|^{-1/(n+2)}g_{ij} = -\hnabla_{i}\cn_{j} - \muf_{i}\cn_{j},
\end{split}
\end{align}
and $\tilde{g}^{ij}$ is the symmetric bivector inverse to $\tilde{g}_{ij}$. Moreover, $d\cn = \cn \wedge \muf$.
\end{lemma}
The co-orientation convention in Lemma \ref{affnormallemma} is such that for a locally uniformly convex hypersurface the equiaffine normal points to the convex (interior) side of the surface.
\begin{proof}
Let $h$ and $\tau$ be respectively the second fundamental form and connection one-form on $\lc_{r}(F, \Om)$ determined by the transversal $\gn^{i} = \H(F)F^{i}$, and let $\nm = a(N + Z)$ be the equiaffine normal. For $X$ tangent to $\lc_{r}(F, \Om)$, $0 = dF(X) = (\hnabla_{X}dF)(N) = d\U(F)(X) - \U(F)\tau(X)$, so $\tau_{I} = \U(F)^{-1}\U(F)_{I} = (n+2)\muf_{I}$. Note that this equality refers only to the tangential directions. For $X$ and $Y$ tangent to $\lc_{r}(F, \Om)$, $g(X, Y) = (\hnabla_{X} dF)(Y) = -\U(F)h(X, Y)$, so that $h_{ij} =  -\U(F)^{-1}\left(g_{ij} - \H(F)\U(F)^{-1}F_{i}F_{j}\right)$ satisfies $F^{p}h_{ip} = 0$ and restricts on $\lc_{r}(F, \Om)$ to the second fundamental determined by $N^{i}$. Let $X_{1}, \dots, X_{n}$ be $h$-orthogonal vector fields tangent to $\lc_{r}(F, \Om)$ such that $|h(X_{I}, X_{I})| = 1$, so that $|\vol_{h}(X_{1}, \dots, X_{n})| = 1$. Since $|g(X_{I}, X_{I})| = |\U(F)|$ and $|g(N, N)| = |\H(F)\U(F)|$, by the definition of the volume density $\vol_{g}$ of the metric $g_{ij}$,
\begin{align}\label{quf}
\begin{split}
|\Psi(\gn, X_{1}, \dots X_{n})| & = |\H(F)|^{-1/2}|\vol_{g}(N, X_{1}, \dots, X_{n})| = |\U(F)|^{(n+1)/2},
\end{split}
\end{align}
so that $q = |\vol_{h}/\imt(N)\Psi|  =  |\U(F)|^{-(n+1)/2}$ and $2q^{-1}dq_{I} = -(n+1)\tau_{I}$. In \eqref{zdet} this yields $Z^{P}h_{PI} = -(n+2)^{-1}\tau_{I} = -\muf_{I}$ and $|a|  = |\U(F)|^{-(n+1)/(n+2)}$. Since $Z$ is tangent to $\lc_{r}(F, \Om)$, for $X$ tangent to $\lc_{r}(F, \Om)$ there holds $g(Z, X) = -\U(F)h(Z, X) = (n+2)^{-1}X^{I}\U(F)_{I}$. Consequently the vector field $Z^{i}  = \U(F)(g^{ij} - \U(F)^{-1}\H(F)F^{i}F^{j})\muf_{j}$ equals $Z^{I}$ along each level set $\lc_{r}(F, \Om)$, and it follows that the equiaffine normal on $\lc_{r}(F, \Om)$ consistent with the co-orientation determined by $-\sign(\H(F)\U(F))F^{i}$ is given by the first equality of \eqref{affnorm}. The claims involving $\tilde{g}_{ij}$ follow straightforwardly. Skew-symmetrizing \eqref{hnablacn} shows $d\cn = \cn \wedge \muf$.
\end{proof}

\begin{lemma}\label{fmclemma}
With the setup as in Lemma \ref{affnormallemma}, the mean curvature of $\lc_{r}(F, \Om)$ with respect to $\hnabla$ and the transversal $F^{i}$ equals $-n^{-1}(n + 2)\A(F)$.
\end{lemma}
\begin{proof}
Define $\tilde{\tau}_{i} = |dF|^{-2}_{g}\left(d_{i}|dF|^{2}_{g} - F_{i}\right)$ and
\begin{align}\label{tauextended}
\begin{split}
\tilde{S}_{i}\,^{j} & = -\hnabla_{i}F^{j} + \tilde{\tau}_{i}F^{j} = -\delta_{i}\,^{j} + F_{ip}\,^{j}F^{p} + |dF|^{-2}_{g}\left(d_{i}|dF|^{2}_{g} -  F_{i}\right)F^{j}\\
&=  -\delta_{i}\,^{j} + F_{ip}\,^{j}F^{p} - \left(d_{i}\log\H(F) - (n+2)\muf_{i} + |dF|_{g}^{-2}F_{i} \right)F^{j}.
\end{split}
\end{align}
From $d_{i}|dF|_{g}^{2} = 2F_{i} - F_{ipq}F^{p}F^{q}$, it follows that $F_{j}\tilde{S}_{i}\,^{j} = 0$, so it makes sense to speak of the restriction $\tilde{S}_{I}\,^{J}$ of $\tilde{S}_{i}\,^{j}$ to the tangent bundle of $\lc_{r}(F, \Om)$, and  $\tilde{S}_{I}\,^{I} = \tilde{S}_{i}\,^{i}$. It follows from \eqref{tauextended} that $\tilde{S}_{I}\,^{J}$ is the shape operator associated with the transversal $F^{i}$. Hence the mean curvature with respect to $F^{i}$ is $\tilde{S}_{I}\,^{I} = \tilde{S}_{p}\,^{p} = -(n+2) + F^{p}d_{p}\log\H(F) + F^{p}d_{p}\log|dF|^{2}_{g} = -(n+2)\A(F)$.
\end{proof}

By \eqref{hpsif}, the condition that $\U(F)$ be a function of $F$ is preserved when $F$ is replaced by an external reparameterization $\psi \circ F$ by a $C^{1}$ function $\psi$. In short, this is really a condition about the geometry of the level sets of $F$, rather than about $F$ per se. Lemma \ref{uflemma} shows that the local constancy of $\U(F)$ on the levels of $F$ is equivalent to the preservation of the levels of $F$ by the flow generated by their affine normals.

\begin{lemma}\label{uflemma}
Suppose given connected open subsets $\Om \subset \rea^{n+1}$ and $I \subset \rea$ and a function $F \in \cinf(\Om)$ such that neither $\H(F)$ nor $\U(F)$ vanishes on $\Om_{I} = F^{-1}(I) \cap \Om$. Then the following are equivalent: 
\begin{enumerate}
\item\label{uf1} $\U(F)$ is locally constant on $\lc_{r}(F, \Om_{I})$ for each $r \in I$. 
\item\label{ufnormal} The equiaffine normal $\nm^{i}$ is a multiple of $F^{i}$.
\item\label{ufcn} The equiaffine conormal one-form $\cn$ is closed, $d\cn = 0$.
\item\label{ufliecn} There holds $\lie_{\nm}\cn = 0$.
\item\label{uflie} The flow generated by the equiaffine normal preserves the level sets of $F$. 
\end{enumerate}
If there hold \eqref{uf1}-\eqref{uflie}, then the affine mean curvature of $\lc_{r}(F, \Om_{I})$ is the restriction of
\begin{align}\label{ufmc}
n^{-1}(n+2)\sign(\U(F))\H(F)|\U(F)|^{-(n+1)/(n+2)}\A(F),
\end{align}
where $\A(F)$ is defined in Lemma \ref{affnormallemma}, and the tensor $S_{i}\,^{j}$ defined by
\begin{align}\label{ufshape}
\begin{split}
S_{i}\,^{j} &= -\hnabla_{i}\nm^{j} +\sign(\U(F))\U(F)^{-(n+1)/(n+2)}\H(F)\A(F)\cn_{i}\nm^{j}\\
& = \sign(\U(F))\H(F)|\U(F)|^{-\tfrac{n+1}{n+2}}\left(\delta_{i}\,^{j} - F_{ip}\,^{j}F^{p} 
\right.\\ &\qquad \left. 
+\left(d_{i}\log\H(F) +\U(F)^{-1}\H(F)\left( (n+2)\A(F) - (n+1)\right)F_{i}\right)F^{j} \right),
\end{split}
\end{align}
satisfies $F_{j}S_{i}\,^{j} = 0$, so  restricts along $\lc_{r}(F, \Om_{I})$ to the equiaffine shape operator of $\lc_{r}(F, \Om_{I})$. 
\end{lemma}

\begin{proof}
By the assumption that $\U(F)$ does not vanish on $\Om_{I}$, $F^{i}$ is transverse to $\lc_{r}(F, \Om)$ for $r \in I$. The flow of a smooth vector field $X$ preserves a codimension one smooth foliation of an open manifold $M$ if and only if $\be \wedge \lie_{X}\be = 0$ for any one-form $\be$ annihilating the tangent distribution of the foliation. In particular, \eqref{uflie} is equivalent to $\cn \wedge \lie_{\nm}\cn = 0$. By the definition of $\cn$, \eqref{uf1} holds if and only if $d\cn = \muf \wedge \cn = 0$, and by \eqref{affnorm}, this holds if and only if $\nm^{i} = -\sign(\U(F))\H(F)|\U(F)|^{-(n+1)/(n+2)}F^{i}$. This shows \eqref{uf1}$\iff$\eqref{ufcn}$\iff$\eqref{ufnormal}. From $d\cn = \cn \wedge \muf$ there follow $\lie_{\nm}\cn = \imt(\nm)d\cn = \muf - \muf(\nm)\cn$ and $\cn \wedge \lie_{\nm}\cn = \cn \wedge \muf = d\cn$, which yield the implications \eqref{ufcn}$\implies$\eqref{ufliecn}$\implies$\eqref{uflie}$\implies$\eqref{ufcn}. Lemma \ref{fmclemma} gives the mean curvature of a level set of $F$ with respect to the transversal $F^{i}$. Rescaling a transversal by $e^{f}$ multiplies the corresponding shape operator by $e^{f}$. With this observation, \eqref{ufnormal} and Lemma \ref{fmclemma} imply that the affine mean curvature of $\lc_{r}(F, \Om_{I})$ is \eqref{ufmc}. Likewise, applying this observation to \eqref{tauextended}, and simplifying using $\muf_{i} = |dF|^{-2}_{g}(1 - \A(F))F_{i}$ yields the second equality of \eqref{ufshape}. That $F_{j}S_{i}\,^{j} = 0$ follows from $F_{j}\tilde{S}_{i}\,^{j} = 0$. 
The first equality of \eqref{ufshape} follows from \eqref{affnorm} and \eqref{hnablacn}.
\end{proof}

A vector field $X$ is \textit{projectively geodesic} with respect to the torsion-free affine connection $\nabla$ if its integral curves are projective (unparameterized) geodesics of $\nabla$, that is, if $X\wedge \nabla_{X}X= 0$. If $\tilde{X} = cX$ and $\tnabla = \nabla + 2\ga_{(i}\delta_{j)}\,^{k}$ then $\tilde{X}\wedge \tnabla_{\tilde{X}}\tilde{X} = c^{3}X \wedge \nabla_{X}X$, so this condition depends only on the span of $X$ and the projective equivalence class $\en$ of $\nabla$. Consequently, it makes sense to say that a line field is \textit{projectively geodesic} with respect to $\en$ if any local section $X$ of the line field is projectively geodesic with respect to any representative $\nabla \in \en$. For example, for $F$ with $dF$ and $\H(F)$ nonvanishing, the line field spanned by $F^{i}$ is always projectively geodesic with respect to the $g$-conjugate connection $\bnabla = \nabla + F_{ij}\,^{k}$. That the affine normal distribution be projectively geodesic with respect to $\hnabla$ means that the images of the integral curves of $\nm$ are contained in straight lines. 

\begin{lemma}\label{spherechartlemma}
Suppose given connected open subsets $\Om \subset \rea^{n+1}$ and $I \subset \rea$ and a function $F \in \cinf(\Om)$ such that neither $\H(F)$ nor $\U(F)$ vanishes on $\Om_{I} = F^{-1}(I) \cap \Om$, and $\U(F)$ is locally constant on $\lc_{r}(F, \Om_{I})$ for $r \in I$. Let $D$ be the Levi-Civita connection of $g_{ij}$. Then the following are equivalent:
\begin{enumerate}
\item\label{ug0} The affine normal distribution is projectively geodesic with respect to $\hnabla$.
\item\label{ug1} The affine normal distribution is projectively geodesic with respect to $D$.
\item\label{ug1b} The affine normal distribution is projectively geodesic with respect to every member of the affine line in the space of of connections passing through $\hnabla$ and $D$.
\item\label{ug2} $\H(F)$ is locally constant on the level sets $\lc_{r}(F, \Om_{I})$ for $r \in I$. 
\end{enumerate}
\end{lemma}

\begin{proof}
By Lemma \ref{uflemma} the affine normal distribution is spanned by $F^{i}$ and there holds \eqref{ufshape}, so for $X^{i}$ tangent to $\lc_{r}(F, \Om_{I})$ there holds 
\begin{align}\label{uga}
\begin{split}
0 &= \sign(\U(F))\H(F)^{-1}|\U(F)|^{(n+1)/(n+2)}F_{j}X^{i}\hnabla_{i}\nm^{j}\\& = F_{j}\left(X^{j} - F_{pq}\,^{j}X^{q}F^{q} + X^{p}d_{p}\log\H(F)F^{j}\right) = \left(|dF|^{2}_{g}d_{j}\log\H(F) - F_{pqj}F^{p}F^{q}\right)X^{j},\\
&= \left(|dF|^{2}_{g}d_{j}\log\H(F) + 2F^{p}D_{p}F_{j}\right)X^{j},
\end{split}
\end{align}
in which the last equality follows from $2F^{p}D_{p}F_{i} = d_{i}|dF|_{g}^{2} = 2F_{i} - F_{ipq}F^{p}F^{q}$. Equation \eqref{uga} implies that $d\log\H(F) \wedge dF = 0$ if and only if $\gf \wedge D_{\gf}\gf = 0$, where $\gf$ means $F^{i}$. This shows the equivalence of \eqref{ug1} and \eqref{ug2}. By the affine line of connections determined by $\hnabla$ and $D$ is meant the one-parameter family of affine connections $\tnabla = \nabla + tF_{ij}\,^{k}$. The cases $ t= 1/2$ and $t = 1$ yield, respectively, $D$ and the connection $\bnabla$ that is $g$-conjugate to $\nabla$. Since $F^{p}\tnabla_{q}F^{i} = F^{i} + (t - 1)F^{p}F^{q}F_{pq}\,^{i}$, for $t \neq 1$ there holds $\gf \wedge \tnabla_{\gf}\gf = 0$ if and only if there holds $F_{[i}F_{j]pq}F^{p}F^{q} = 0$. Since $\gf$ is always projectively geodesic for $\bnabla$, this shows the equivalence of \eqref{ug0}, \eqref{ug1}, and \eqref{ug1b}. 
\end{proof}

Say that $F$ is \textit{locally auto-isoparametric} on $\Om_{I}$ if $\lap_{g}F = g^{ij}D_{i}F_{j}$ and $|dF|_{g}^{2}$ are locally constant on $\lc_{r}(F, \Om_{I})$ for each $r \in I$; that is, $F$ is an isoparametric function with respect to $g_{ij}$.

\begin{lemma}\label{autoisolemma}
Suppose given connected open subsets $\Om \subset \rea^{n+1}$ and $I \subset \rea$ and a function $F \in \cinf(\Om)$ such that neither $\H(F)$ nor $\U(F)$ vanishes on $\Om_{I} = F^{-1}(I) \cap \Om$. Suppose that $\U(F)$ and $\H(F)$ are locally constant on $\lc_{r}(F, \Om_{I})$ for each $r \in I$. Then $F$ is locally auto-isoparametric on $\Om_{I}$. 
\end{lemma}
\begin{proof}
The hypotheses mean $d\log \H(F) \wedge dF = 0$ and $\muf \wedge dF = 0$, so, by \eqref{ufdet}, $d|dF|^{2}_{g} \wedge dF = 0$. Since $d_{i}|dF|^{2}_{g} = 2F_{i} - F_{ipq}F^{p}F^{q}$, there is a smooth function $b$ on $\Om_{I}$ such that $F_{ipq}F^{p}F^{q} = bF_{i}$. Differentiating $\lap_{g}F = n+1 - (1/2)F^{p}d_{p}\log\H(F)$ yields
\begin{equation}\label{dlapf}
-2d_{i}\lap_{g}F = d_{i}\log\H(F) + F^{p}\hnabla_{i}d_{p}\log\H(F) - F_{ip}\,^{q}F^{p}d_{q}\log\H(F).
\end{equation}
There is a smooth function $a$ on $\Om_{I}$ such that $d_{i}\log\H(F) = a F_{i}$. Since $\hnabla_{i}d_{j}\log\H(F) = ag_{ij} + a_{i}F_{j}$ is symmetric there is a smooth function $c$ on $\Om_{I}$ such that $a_{i} = cF_{i}$. In \eqref{dlapf} this yields
\begin{equation}
-2d_{i}\lap_{g}F = aF_{i} + F^{p}(ag_{ip} + cF_{i}F_{p}) - aF_{ipq}F^{p}F^{q} = (a(2-b) + c|dF|^{2}_{g})F_{i},
\end{equation}
which proves that $dF\wedge d\lap_{g}F = 0$.
\end{proof}

The remainder of this section is devoted to the proof of Theorem \ref{ahtheorem}. Let $\la, \al \in \rea$ and $\ep \in \{0, 1\}$ and suppose $\la\al = 0$, $\ep\al = 0$, and $\la \neq 1-\ep$. A function $F \in \cinf(\Om \setminus \lc_{-\al/\la}(F, \Om)) \cap C^{0}(\Om) $ is in $\amg^{\la, \ep}_{\al}$ if and only if there is a $v \in \rea^{n+1}$ such that $\rad^{p}F_{p} = \la F + \al$, where $\rad^{i} = (1-\ep)\eul^{i} +(2\ep -1)v^{i}$. Since $\hnabla_{i}\rad^{j} = (1-\ep)\delta_{i}\,^{j}$, differentiating $\rad^{p}F_{p} = \la F + \al$ yields $\rad^{p}F_{pi_{1}\dots i_{k}} = (\la +k(\ep -1))F_{i_{1}\dots i_{k}}$. Then, for $F \in \amg^{\la, \ep}_{\al}$ there hold 
\begin{align}\label{hompol1}
&(\la - 1 + \ep)F^{i}= \rad^{i},&
&(\la - 1 + \ep)\U(F)  = (\la F + \al)\H(F).
\end{align}
Tracing $\rad^{p}F_{ijp} = (\la - 2 + 2\ep)F_{ij}$ and combining the result with \eqref{hompol1} yields
\begin{align}\label{hompol2}
&\rad^{p}\H(F)_{p} = \H(F)\rad^{p}F_{pq}\,^{q} = (n+1)(\la - 2 +2\ep)\H(F), &
&\A(F) = \tfrac{n(1-\ep)}{(n+2)(\la -1 + \ep)}.
\end{align}

\begin{lemma}\label{homognondegenlemma}
Given an open domain $\Om \subset \rea^{n+1}$ and constants $\la, \al, r \in\rea$ and $\ep \in \{0, 1\}$ such that $\la \al = 0$, $\ep\al = 0$, and $\la r + \al \neq 0$, the level set $\lc_{r}(F, \Om)$ of $F \in \amg^{\la, \ep}_{\al}(\Om)$ is smoothly immersed and transverse to $\rad^{i}= (1-\ep)\eul^{i} +(2\ep -1) v^{i}$ for some $v \in \rea^{n+1}$. If, moreover, $\la \neq 1-\ep$, then the level set $\lc_{r}(F, \Om)$ is nondegenerate if and only if $\H(F)$ does not vanish on $\lc_{r}(F, \Om)$, in which case $dF$ and $\U(F)$ do not vanish on $\lc_{r}(F, \Om)$, and the equiaffine normal $\nm^{i}$ of $\lc_{r}(F, \Om)$ is
\begin{equation}\label{nm2}
\nm^{i} =  -\tfrac{1}{(n+2)(\la-1+\ep)}\left|\tfrac{(\la F + \al)\H(F)}{\la - 1+\ep}\right|^{1/(n+2)}\left(\tfrac{n(1-\ep)+\la}{\la F + \al }\rad^{i} + (\la - 1+\ep)(d\log\H(F))^{i}\right).
\end{equation}
\end{lemma}

\begin{proof}
By assumption there is $v \in \rea^{n+1}$ such that $\rad^{i}F_{i}(x) = \la F(x) + \al$, where $\rad^{i} = (1-\ep)\eul^{i} +(2\ep -1) v^{i}$. Since for $x \in \lc_{r}(F, \Om)$, $\rad^{i}F_{i}(x) = \la r + \al \neq 0$, $dF$ does not vanish on $\lc_{r}(F, \Om)$ and so the level set $\lc_{r}(F, \Om)$ is smoothly immersed; moreover, the vector field $\rad$ is transverse to $\lc_{r}(F, \Om)$. Let $h$ be the corresponding second fundamental form. For $X$ and $Y$ tangent to $\lc_{r}(F, \Om)$ there hold 
\begin{align}\label{hdl}
\begin{split}
&\hess F(X, Y) = (\hnabla_{X}dF)(Y) = -dF(\rad)h(X, Y) = -(\la r + \al)h(X, Y), \\
&\hess F(X, \rad) = (\la-1+\ep)dF(X) = 0,\\
&\hess F(\rad, \rad) = (\hnabla_{\rad}dF)(\rad) = (\la-1+\ep)(\la r + \al),
\end{split}
\end{align}
along $\lc_{r}(F, \Om)$. If $\la \neq 1-\ep$ and $\la r + \al\neq 0$, it follows from \eqref{hdl} that $h$ is nondegenerate if and only if $\hess F$ is nondegenerate on $\lc_{r}(F, \Om)$, or, equivalently, $\H(F)$ does not vanish along $\lc_{r}(F, \Om)$. In this case, since $\H(F)$ does not vanish on $\lc_{r}(F, \Om)$ it follows from \eqref{hompol1} and $\la r + \al \neq 0$ that neither $dF$ nor $\U(F)$ vanishes there. Substituting \eqref{hompol1} and \eqref{hompol2} into \eqref{affnorm} yields \eqref{nm2}.
\end{proof}

\begin{proof}[Proof of Theorem \ref{ahtheorem}]
The equivalence of \eqref{aht2} and \eqref{aht3} of Theorem \ref{ahtheorem}, and that in this case that $\psi$ is related to $\phi$ as indicated in \eqref{phiforms}, are immediate from \eqref{hompol1} and \eqref{hompol2}. 

Suppose there holds \eqref{aht1proper} or \eqref{aht1improper} of Theorem \ref{ahtheorem}. That is, $F \in \amg^{\la, \ep}_{\al}(\Om)$ and there is a connected open interval $I \subset \rea \setminus\{-\al/\la\}$ such that for all $r \in I$ each level set $\lc_{r}(F, \Om_{I})$, equipped with the co-orientation of its components consistent with $-\sign(|dF|^{2}_{g})F^{i}$, is an affine sphere with center $v$ and affine mean curvature $\amc$, if $\ep = 0$, or is an affine sphere with affine normal parallel to $v$, if $\ep = 1$. Because by assumption each connected component of $\lc_{r}(F, \Om_{I})$ is nondegenerate, Lemma \ref{homognondegenlemma} implies that neither $\H(F)$ nor $\U(F)$ vanishes on $\lc_{r}(F, \Om_{I})$. A posteriori, this justifies assigning to each component the co-orientation given by $-\sign(|dF|^{2}_{g})F^{i}$.

Suppose $\ep = 0$. By assumption $\nm^{i} = -\amc\rad^{i}$ along $\lc_{r}(F, \Om_{I})$. Comparing with \eqref{nm2} shows that $d_{i}\log\H(F)$ is a multiple of $\rad^{p}g_{pi} = (\la - 1)F_{i}$, which implies that $\H(F)$ is locally constant on $\lc_{r}(F, \Om_{I})$. Contracting $-\amc \rad^{i}$ with $\cn_{i}$ shows that, along $\lc_{r}(F, \Om_{I})$,
\begin{equation}\label{hfrel}
\amc = \sign(\la r + \al)|\la - 1|^{-1/(n+2)}|\la r + \al|^{-(n+1)/(n+2)}|\H(F)|^{1/(n+2)}.
\end{equation}
Consequently, $|\H(F)| = \sign(\la r + \al)|\la - 1|(\la r + \al)^{n+1}\amc^{n+2}$. The sign of $\H(F)$ does not change on each connected component of $\lc_{r}(F, \Om_{I})$. By assumption the signatures of the second fundamental forms of the connected components of $\lc_{r}(F, \Om_{I})$ are the same modulo $4$, and by \eqref{hdl} this implies that the signatures of $\hess F$ on the different connected components are the same modulo $4$, and so the signs of $\H(F)$ on the different connected components must be the same. This means that $|\H(F)|$ can be replaced coherently by one of $\pm \H(F)$ on all of $\lc_{r}(F, \Om_{I})$. Since, also by assumption, the values of $\amc$ on the different connected components are the same, it follows that $\H(F)$ is constant on $\lc_{r}(F, \Om_{I})$. Since this holds for each $r \in I$, there is a function $\phi$ defined on $I$ such that $\H(F) = \phi(F)$ on $\Om_{I}$. This proves the implication \eqref{aht1proper}$\implies$\eqref{aht2} of Theorem \ref{ahtheorem}. Now suppose $\ep = 1$. Then, by assumption, $\nm \wedge v = 0$ along $\lc_{r}(F, \Om_{I})$. Comparing with \eqref{nm2} shows that $d_{i}\log\H(F) = cv_{i} = c\la F_{i}$ for some $c$ locally constant on $\lc_{r}(F, \Om_{I})$. Contracting with $v^{i} = \la F^{i}$ and using \eqref{hompol2} yields $(n+1)\la = c\la^{2}F$, so that $d_{i}\log\H(F) = (n+1)F^{-1}F_{i}$. Hence $F^{-n-1}\H(F)$ is locally constant on $\lc_{r}(F, \Om_{I})$. As in the $\ep = 0$ case the assumption that the signatures of the second fundamental forms of the connected components of $\lc_{r}(F, \Om_{I})$ are the same modulo $4$ implies that the sign of $\H(F)$ is constant on $\lc_{r}(F, \Om_{I})$. Since, by assumption, $\nm^{i} = cv^{i}$ for $c$ depending only on $r$ and not the connected component of $\lc_{r}(F, \Om_{I})$, it follows from \eqref{nm2} that $\H(F)$ is constant on $\lc_{r}(F, \Om_{I})$.

The implications \eqref{aht2}$\implies$\eqref{aht1proper} and \eqref{aht2}$\implies$\eqref{aht1improper} of Theorem \ref{ahtheorem} are proved as follows. Suppose $\la \neq 1 - \ep$ and $\la r + \al \neq 0$ for $r \in I$. If $F \in \amg^{\la, \ep}_{\al}(\Om)$ solves $\H(F) = \phi(F)$ on $\Om_{I}$ for some nonvanishing function $\phi:I \to \rea$, then by Lemma \ref{homognondegenlemma}, each level set $\lc_{r}(F, \Om_{I})$ is nondegenerate  and $dF$ and $\U(F)$ do not vanish on $\lc_{r}(F, \Om_{I})$. In particular, the equiaffine normal $\nm^{i}$ is defined on $\Om_{I}$. Since $\H(F)$ is constant on $\lc_{r}(F, \Om_{I})$, there holds $d\log\H(F) \wedge dF = 0$ on $\Om_{I}$. Hence, by \eqref{hompol1}, $(d\log \H(F))^{i}$ is a multiple of $F^{i} = (\la - 1 + \ep)^{-1}\rad^{i}$. In \eqref{nm2} this shows that $\nm^{i}$ is a multiple of $\rad^{i}$, so that the connected components of $\lc_{r}(F, \Om_{I})$ are affine spheres with a common center $v$, if $\ep = 0$, or with affine normals parallel to $v$, if $\ep = 1$. In the case $\ep = 0$, there holds \eqref{hfrel} on each component of $\lc_{r}(F, \Om_{I})$. Since by assumption $\H(F) = \phi(F)$ depends only on $r$, and not on the component, it follows that the affine mean curvatures of different components of $\lc_{r}(F, \Om_{I})$ are the same. In both cases, from the constancy of $\H(F)$ on each $\lc_{r}(F, \Om_{I})$ and \eqref{hdl} it follows that the signatures of the distinct connected components of $\lc_{r}(F, \Om_{I})$ are the same modulo $4$. 

Suppose given $F \in \amg^{\la, \ep}_{\al}(\Om)$, an open interval $I \subset \rea \setminus\{-\al/\la\}$, and a function $\phi$ defined on $I$ such that $\H(F) = \phi(F)$ for $x \in \Om_{I}$. Since, by \eqref{hompol2}, $\H(F)$ has positive homogeneity $(n+1)(\la - 2 + 2\ep)$, there holds $\phi(e^{\la t}r + \al t) = e^{(n+1)(\la-2 + 2\ep)t}\phi(r)$ for $r \in I$ and $t$ sufficently small. In particular, this shows that $\phi$ is continuous on $I$. Similarly, setting $h(t) = (e^{\la t} - 1) r + \al t$,
\begin{equation}
\lim_{t \to 0} \tfrac{\phi(r + h(t)) - \phi(r)}{h(t)} 
= \lim_{t \to 0} \tfrac{\phi(e^{\la t}r + \al t) - \phi(r)}{(\la r + \al)t} = \lim_{t \to 0}\tfrac{(e^{(n+1)(\la -2 + 2\ep)t} - 1)}{(\la r+ \al)t}\phi(r)= \tfrac{(n+1)(\la -2 + 2\ep)}{\la r + \al}\phi(r),
\end{equation}
so that $\phi$ is differentiable at $r$ and $\phi^{\prime}(r) =  \tfrac{(n+1)(\la -2+2\ep)}{\la r + \al}\phi(r)$. 
When $\la \neq 0$, then $\al = 0$, and the general solution has the form $\phi(r) = B| r|^{(n+1)(\la - 2 + 2\ep)/\la}$ for some $B \neq 0$. If $\ep =1$, then this can be written $\phi(r) = Br^{n+1}$ for some $B \neq 0$. When $\la = 0$ and $\ep = 0$ the general solution has the form $Be^{-2(n+1)r/\al}$ for a nonzero constant $B$. Hence $\phi$ must have one of the forms \eqref{phiforms}. That the affine mean curvature of $\lc_{r}(F, \Om_{I})$ has the form \eqref{ahmc} follows upon substituting \eqref{phiforms} in \eqref{hfrel}.
\end{proof}

\section{Relative invariants of prehomogeneous vector spaces}\label{prehomogeneoussection}
The theory of prehomogeneous vector spaces is due to M. Sato. The background recalled below is culled from \cite{Sato-Kimura}, \cite{Sato-Shintani}, and \cite{Kimura}. A triple $\pv$ comprising a complex vector space $\ste$, a connected complex linear algebraic group $G$, and a rational representation $\rho:G \to GL(\ste)$ is a \textit{prehomogeneous vector space} if there is a proper algebraic subset $\sing \subset \ste$, called the \textit{singular set}, such that $\ste \setminus \sing$ is a single (necessarily Zariski dense) $G$ orbit. While this definition can be extended to a not algebraically closed field such as $\rea$, here it is more convenient to speak of a real form of a prehomogeneous vector space than it is to speak directly of real prehomogeneous vector spaces. 

A nonzero rational function $Q$ on $\ste$ is a \textit{relative invariant} of $\pv$ \textit{corresponding} to a rational character $\chi:G \to \comt$ if $\rho(g)\cdot Q = \chi(g^{-1})Q$ for all $g \in G$. There hold the following for relative invariants of a prehomogeneous vector space (see Proposition $2$ in \cite{Sato-Shintani}):
\begin{enumerate}
\item Any two relative invariants corresponding to a given rational character are related by multiplication by a nonzero scalar.
\item Any prime divisor of a relatively invariant polynomial is again a relative invariant.
\item A relative invariant is a homogeneous function.
\end{enumerate}

There is a relative invariant corresponding to the rational character $\chi$ if and only if $\ker \chi$ contains the stabilizer in $G$ of some (and hence any) $v \in \ste \setminus \sing$ (Proposition $4.19$ of \cite{Sato-Kimura}). There exists a relative invariant of $\pv$ if and only if $\sing$ has an (Zariski) irreducible component of codimension one (Corollary $4.6$ of \cite{Sato-Kimura}). In this case each of the irreducible components $\sing_{1}, \dots, \sing_{l}$ of the codimension one part of $\sing$ is the zero locus of an irreducible polynomial $P_{i}$. The $P_{1}, \dots, P_{l}$ are algebraically independent relative invariants and the multiplicative group of relative invariants is the free abelian group they generate. That is, any relative invariant has the form $\la P_{1}^{n_{1}}\dots P_{l}^{n_{l}}$ for some $\la \in \com$ and $(n_{1}, \dots, n_{l}) \in \integer^{l}$, and the group of characters corresponding to relative invariants is the rank $l$ free abelian group generated by the rational characters $\chi_{i}$ corresponding to the $P_{i}$. In particular, $\pv$ has a relative invariant if and only if $\sing$ has an irreducible component of codimension one. In the case that $\sing$ has a unique codimension one irreducible component, every relative invariant is a constant multiple of an integer power of the irreducible polynomial $P$ defining that component. In this case $P$ is referred to as \textit{the} relative invariant of $\pv$, or the \textit{fundamental} relative invariant.

Given a complex prehomogeneous vector space $\pv$, identify the complex structure on $\ste$ with an almost complex structure $J$, and let $\nabla$ be the flat torsion-free affine connection such that $\nabla J = 0$ and the $(1,0)$ part $\nabla^{(1, 0)}$ of $\nabla$ is the standard flat holomorphic affine connection. Write $d = \pr + \bar{\pr}$ for the decomposition of the exterior differential into its $(1,0)$ and $(0, 1)$ parts. Then the complex algebraic Hessian of a smooth function $F$ is $\nabla^{(1,0)}\pr F$. Define $\H_{\com}(F)$ by $\det \nabla^{(1, 0)}\pr F = \H(F)\Psi^{2}$ where $\Psi$ is the parallel holomorphic volume form defined by the usual determinant, and $\Psi^{2}$ is its tensor square. Thus $\H(F)$ is defined just as it was over the real field, though now over the complex numbers. This construction works over any field of characteristic zero and could be described in purely algebraic terms, although from the perspective adopted here, in speaking of polynomials there is implicitly fixed a flat connection (with respect to which the differentials of the variables constitute a parallel coframe). If $z^{1}, \dots, z^{n}$  ($n = \dim \ste$) are coordinates such that $dz^{1}, \dots, dz^{n}$ constitute a parallel (complex) coframe, then $\H(F) = \det \tfrac{\pr^{2}F}{\pr z^{i}\pr z^{j}}$ is the determinant of the matrix of second partial (complex) derivatives of $F$. More generally, if $X_{1}, \dots, X_{n}$ is a frame in $\ste$ such that $\Psi(X_{1}, \dots, X_{n})^{2} = 1$, then $\H(F)$ is the determinant of the matrix whose entries are $(\nabla dF)^{(2, 0)}(X_{i}, X_{j})$. If, moreover, $F$ is holomorphic, then $(\nabla dF)^{(2, 0)} = \nabla \pr F$. 
When it is necessary to distinguish the real and complex versions of $\H(F)$ they will be written $\H^{\rea}(F)$ or $\H^{\com}(F)$, but otherwise the superscript will be omitted, as the interpretation should be clear from context.

A prehomogeneous vector space $\pv$ is \textit{regular} if there is a relative invariant $P$ such that $\H(P)$ does not vanish identically. In this case $\H(P)$ is also a relative invariant, with the character $(\det^{-2} \rho) \tensor \chi^{\dim \ste}$ (where $\chi$ is the character of $P$), for if $g \in G$ then
\begin{align}\label{hprel}
\rho(g)\cdot \H(P) =  (\det\rho(g))^{2}\H(\rho(g) \cdot P) = (\det\rho(g))^{2}\H(\chi(g^{-1})P) = (\det\rho(g))^{2}\chi(g)^{-\dim \ste}\H(P). 
\end{align} 
Since $\H(P)$ is a relative invariant, it has the form $\la P_{1}^{n_{1}}\dots P_{l}^{n_{l}}$ with nonnegative exponents not all zero, and so it does not vanish on $\ste \setminus \sing$. The triple $\pv$ is \textit{irreducible} if $\rho$ is irreducible. If $\pv$ is irreducible then, by Proposition $4.12$ of \cite{Sato-Kimura}, there is up to multiplication by a constant factor at most one irreducible relative invariant polynomial $P$. In this case the degree $k$ of the fundamental relative invariant $P$ divides $2\dim\ste$, and it follows from \eqref{hprel} that the corresponding $\chi$ satisfies $\chi^{2\dim\ste/k} = (\det \rho)^{2}$. In particular, a regular irreducible prehomogeneous vector space has a unique (up to multiplication by a nonzero constant) irreducible relatively invariant polynomial $P$. What is essential here is the following consequence. Since $\H(P)$ is also relatively invariant, there must hold $\H(P) = \kc P^{m}$ for some $\kc \in \comt$ and some nonnegative integer $m$. Note, however, that solutions of \eqref{ma} have not yet been obtained, since in general $P$ has complex coefficients.

A \textit{real structure} on a complex vector space $\ste$ means a nontrivial antilinear involution $\tau :\ste \to \ste$. The \textit{real points} $\ste_{\rea}$ are those elements of $\ste$ fixed by $\tau$. There results an identification $\ste \simeq \ste_{\rea}\tensor_{\rea}\com$. A real structure $\tau$ on $\ste$ induces a real structure on $\eno(\ste)$ defined by $\tau(\psi) = \tau \circ \psi \circ \tau$ for $\psi \in \eno(\ste)$. For the definition of a real structure on a complex algebraic variety, and the corresponding notion of a real form of a complex linear algebraic group, see \cite{Borel-algebraicgroups}. A \textit{real form} of $\pv$ means a triple $\pvr$ comprising a real form $G_{\rea}$ of $G$ and a real form $\ste_{\rea}$ of $\ste$ such that $\rho:G \to GL(\ste)$ is an $\rea$-rational representation. In particular, $\rho(G_{\rea}) \subset \rho(G)_{\rea} \subset GL(\ste_{\rea})$. Equivalently, there are given real structures $\si$ and $\tau$ on $G$ and $\ste$, respectively, compatible with $\rho$ in the sense that $\rho(\si(g)) = \tau(\rho(g)) = \tau \circ \rho(g) \circ \tau$ for all $g \in G$, where $\tau(\rho(g))$ indicates the action on $\rho(g)$ of the induced real structure on $\eno(\ste)$, so that $\rho(\si g)\cdot v = \tau(\rho(g))\cdot v = \tau(\rho(g)\cdot \tau(v))$. In this case the fixed point sets of $\si$ and $\tau$ are written $G_{\rea}$ and $\ste_{\rea}$, and are called the real points of $G$ and $\ste$, and $\pv$ is said to be defined over $\rea$. It is common in the literature to speak of the triple $(\Grea^{+}, \rho, \ste)$, where $\Grea^{+}$ is the connected component of the identity in $\Grea$, as a \textit{real prehomogeneous vector space}. Although this has the virtue that it is the dense $\Grea^{+}$ orbit that is most directly the real form of the dense $G$ orbit, here reference will only be made to the real form $\pvr$ as just defined.

As $G$ is by assumption linear algebraic, a complex prehomogeneous vector space $\pv$ can also be regarded as a real representation of the real Lie group underlying $G$ on the real vector space underlying $\ste$. For an irreducible real representation of a real Lie algebra either its complexification is an irreducible complex representation or it is the real representation underlying an irreducible complex representation admitting no invariant complex structure. It follows that, modulo connectedness issues, an irreducible real prehomogeneous vector space is either a real form of an irreducible complex prehomogeneous vector spaces, or the real prehomogeneous vector space underlying a complex prehomogeneous vector space. In either case, the relative invariant of this associated complex prehomogeneous vector space yields the sought after solution of \eqref{ma}, in the manner stated in Theorem \ref{pvtheorem}. Lemma \ref{complexsimplelemma} was communicated to the author by Roland Hildebrand. 

\begin{lemma}[R. Hildebrand]\label{complexsimplelemma}
Let $P$ be a degree $k$ homogeneous complex polynomial $P$ on the $(n+1)$-dimensional complex vector space $\ste$ and suppose $P$ solves $\Hc(P) = \kc P^{m}$ for some $\kc \in \comt$, where $m = (n+1)(k-2)/k$. Then the degree $2k$ homogeneous real polynomial $|P|^{2} = P\bar{P}$ solves
\begin{equation}\label{macomplex}
\H^{\rea}(|P|^{2}) = (-4)^{n+1}|\kc|^{2}\tfrac{1-2k}{(k-1)^{2}}|P|^{2(n+1 + m)}.
\end{equation}
Here the operator $\Hc$ is defined with respect to the standard holomorphic volume form $\Phi$ on $\ste$, while the operator $\H^{\rea}$ is defined with respect to the real volume form $\Psi = (-2\j)^{-n-1}\Phi \wedge \bar{\Phi}$.
\end{lemma}

\begin{proof}
Let $\nabla$ be the standard flat holomorphic affine connection on $\ste$ preserving $\Phi$. Let $\eul$ be the position vector in $\ste_{\rea}$, so that $dP(\eul) = \pr P(\eul^{(1, 0)}) = kP$. Let $X_{1}, \dots, X_{n}$ be vectors in $\ste_{\rea}$ that with $\eul$ span $\ste$ over $\com$ and satisfy $\Phi(\eul^{(1, 0)}, X_{1}^{(1, 0)}, \dots, X_{n}^{(1, 0)}) = 1$, and write $X_{0} = \eul$.
Since $P$ is holomorphic, $\nabla dP = \nabla \pr P$ has type $(2, 0)$. By definition $\Hc(P)$ is the determinant of the matrix with entries $(\nabla \pr P)(X_{i}^{(1, 0)}, X_{j}^{(1, 0)}) = (\nabla \pr P)(X_{i}, X_{j})$, $0 \leq i, j \leq n$. Hence $\Hc(P)$ equals the product of $k(k-1)P = (\nabla dP)(\eul, \eul)$ with $\det (\nabla \pr P)(X_{I}, X_{J})$, where $1 \leq I, J \leq n$. That is
\begin{equation}\label{cs1}
\det (\nabla \pr P)(X_{I}, X_{J}) = \Hc(P)/(\nabla dP)(\eul, \eul) = \tfrac{\kc}{k(k-1)}P^{m-1}.
\end{equation}
By definition $\Psi(X_{0}^{(1,0)}, \dots, X_{n}^{(1,0)}, X_{0}^{(0,1)}, \dots, X_{n}^{(0, 1)}) = (-2\j)^{-n-1}$, so $(-4)^{-n-1}\H^{\rea}(|P|^{2})$ is the determinant of the matrix with four blocks containing entries of the forms $(\nabla d|P|^{2})(X_{i}^{(1,0)}, X_{j}^{(1, 0)})$,  $(\nabla d|P|^{2})(X_{i}^{(1,0)}, X_{j}^{(0, 1)})$, $(\nabla d|P|^{2})(X_{i}^{(0,1)}, X_{j}^{(1, 0)})$, and  $(\nabla d|P|^{2})(X_{i}^{(0,1)}, X_{j}^{(0, 1)})$, where $0 \leq i, j \leq n$. Because $\bpr P = 0$, $\nabla d|P|^{2} = \bar{P}\nabla^{(1,0)}\pr P + P \nabla^{(0, 1)}\bpr \bar{P} + \pr P \tensor \bpr \bar{P} + \bpr \bar{P} \tensor \pr P$. It can be supposed that $X_{1}^{(1, 0)}, \dots, X_{n}^{(1, 0)}$ span the complex kernel of $dP$. Since $d|P|^{2}(\eul) = 2k|P|^{2}$ and $d|P|^{2}(J\eul) = 0$, 
\begin{align}\label{cs2}
\begin{split}
&(\nabla d |P|^{2})(\eul^{(1, 0)}, \eul^{(1, 0)}) = \bar{P}(\nabla dP)(\eul, \eul) = k(k-1)|P|^{2} = (\nabla d |P|^{2})(\eul^{(0, 1)}, \eul^{(0, 1)}) ,\\
&(\nabla d|P|^{2})(\eul^{(1, 0)}, \eul^{(0, 1)}) = (\nabla d|P|^{2})(\eul^{(0, 1)}, \eul^{(1, 0)}) = k^{2}|P|^{2},\\
&(\nabla d|P|^{2})(X_{I}^{(1, 0)}, X_{J}^{(0, 1)}) = dP(X_{I})d\bar{P}(X_{J}) = 0.
\end{split}
\end{align}
Calculating using \eqref{cs1} and \eqref{cs2} yields
\begin{equation}
\begin{split}
(-4&)^{-n-1}\H^{\rea}(|P|^{2})  = k^{2}(1-2k)|P|^{4}\det (\nabla d |P|^{2})(X_{I}^{(1,0)}, X_{J}^{(1, 0)}) \det (\nabla d |P|^{2})(X_{I}^{(0, 1)}, X_{J}^{(0, 1)})\\
& = k^{2}(1-2k)|P|^{2(n+2)}|\det (\nabla dP)(X_{I}, X_{J})|^{2}  = |\kc|^{2}\tfrac{1-2k}{(k-1)^{2}}|P|^{2(n+1 + m)}.
\qedhere
\end{split}
\end{equation}
\end{proof}

\begin{theorem}\label{pv2theorem}
Let $\pv$ be a regular complex prehomogeneous vector space such that the codimension one part of its singular set is an irreducible hypersurface in $\ste$. Then:
\begin{enumerate}
\item\label{pvv1} Given a real form of $\pv$, the restriction $P$ to the real points $\ste_{\rea}$ of $\ste$ of an appropriate scaling of the relative invariant of $\pv$ is an irreducible homogeneous real polynomial solving \eqref{ma}. A connected component of a nonzero level set of such a $P$ is a homogeneous affine sphere.
\item\label{pvv2} The real polynomial $|P|^{2} = P\bar{P}$ solves \eqref{ma} on $\ste$ viewed as a real vector space. A connected component of a nonzero level set of $|P|^{2}$ is a homogeneous affine sphere.
\end{enumerate}
\end{theorem}

\begin{proof}
The assumption means that $\pv$ has a fundamental relatively invariant polynomial $P$ that solves $\H(P) = \kc P^{m}$ for some $\kc \neq 0$, where $m = (n+1)(k-2)/k$ and $n+1 = \dim \ste$. By Lemma $1.1$ of \cite{Sato-Shintani}, the restriction to the real points $\ste_{\rea}$ of $\ste$ of $P$ is a complex multiple of a polynomial with coefficients in $\rea$, and the corresponding character $\chi$ is defined over $\rea$. Normalizing $P$ to have real coefficients, it is straightforward to see that $\H(P)$ has real coefficients as well, so in the equality $\H(P) = \kc P^{m}$, $\kc$ can be taken to be real. Here $\H(P)$ means the operator $\Hc(P)$ over $\com$. However, since $P$ is holomorphic, $\nabla dP = \nabla^{(1,0)}\pr P$, so if $X_{1}, \dots, X_{n+1}$ ($n+1 = \dim \ste$) is a basis of $\ste_{\rea}$, then $(\nabla dP)(X_{i}, X_{j}) = (\nabla^{(1, 0)}\pr P)(X_{i}, X_{j})$, and so $\H^{\rea}(P) = \H^{\com}(P)$. Hence the restriction to $\ste_{\rea}$ of $P$ solves \eqref{ma}, and so, by Theorem \ref{matheorem}, the connected components of the nonzero level sets of $P$ are affine spheres. By construction the connected component $G_{\rea}^{0}$ of the identity of the real group $G_{\rea}$ acts transitively on each connected component of the complement $\ste_{\rea} \setminus \{x \in \ste_{\rea}: P(x) = 0\}$, each of which is an open cone, and so its unimodular subgroup $SG_{\rea}^{0} = G_{\rea}^{0}\cap SL(n+1, \rea)$ acts transitively on each connected component of a level set of $P$, showing that these level sets are homogeneous. This shows \eqref{pvv1}. By Lemma \ref{complexsimplelemma} the degree $2k$ homogeneous real polynomial $|P|^{2}$ solves \eqref{macomplex}, and so, by Theorem \ref{matheorem}, the connected components of its level sets are affine spheres. The subgroup $SG = \{g \in G: |\chi(g)| = 1\}$, where $\chi$ is the character associated with $P$, acts transitively on each connected component of a level set of $|P|^{2}$. This shows \eqref{pvv2}.
\end{proof}

\begin{proof}[Proof of Theorem \ref{pvtheorem}]
By Proposition $4.12$ of \cite{Sato-Kimura} an irreducible complex prehomogeneous vector space admits (up to multiplication by a nonzero scalar) at most one irreducible relative invariant. Consequently, a regular irreducible complex prehomogeneous vector space has a unique (up to scalars) irreducible relative invariant, and so Theorem \ref{pvtheorem} follows from Theorem \ref{pv2theorem}.
\end{proof}

Different real forms of $\pv$ can give rise to affinely inequivalent $P$. For an example, let $\ste = \com^{3}$ and let $G = GL(1, \com) \times GL(1, \com) \times GL(1, \com)$, acting on $\com^{3}$ by scaling in each coordinate. Then $\sing = \{z \in \com^{3}: P(z) = 0\}$ where $P(z) = z_{1}z_{2}z_{3}$. Note that $\Hc(P) = 2P$. The usual complex conjugation gives a real structure for which the resulting real polynomial is $P(x) = x_{1}x_{2}x_{3}$. Another real structure is given by $\si(z_{1}, z_{2}, z_{3}) = (\bar{z}_{1}, \bar{z}_{3}, \bar{z}_{2})$ and $\si(\rho(g)z) = \rho(\si(g))\si(z)$. The real points $\ste_{\rea}$ have the form $(x_{1}, x_{2} + \j x_{3}, x_{2} - \j x_{3})$ for $x \in \rea^{3}$, and the restriction to $\ste_{\rea}$ of $P$ is $x_{1}(x_{2}^{2} + x_{3}^{2})$, which is not (real) affinely equivalent to $x_{1}x_{2}x_{3}$. It can be shown that any solution $P \in \pol^{3}(\rea^{3})$ of $\H(P) = \kc P$ with $\kc \neq 0$ is affinely equivalent to $x_{1}x_{2}x_{3}$ or $x_{1}(x_{2}^{2} + x_{3}^{2})$, according to the sign of $\kc$. Note that, in conjunction with Lemma \ref{complexsimplelemma}, $P(z) = z_{1}z_{2}z_{3}$ yields a third solution of \eqref{ma}. Namely, the polynomial $|P|^{2} = (x_{1}^{2} + y_{1}^{2})(x_{2}^{2} + y_{2}^{2})(x_{3}^{2} + y_{3}^{2})$ solves $\H(|P|^{2}) = 320|P|^{8}$.

There need not be a real form $\pvr$ of the prehomogeneous vector space $\pv$ corresponding to a real form $\Grea$ of $G$, because it need not be the case that there be a real structure on $\ste$ compatible with that on $G$. Let $d\rho:\g \to \ste$ be the induced representation of the Lie algebra $\g$ of $G$. By a theorem of E. Cartan (Theorem $1.1$ of \cite{Sato-Kimura}) the image $d\rho(\g) \subset \gl(\ste)$ is reductive with one-dimensional center. Let $\Grea$ be a real form of $G$ and let $\g_{\rea}$ be its Lie algebra. Then $d\rho(\g_{\rea})$ is a real form of $d\rho(\Grea)$ and its semisimple part $\s$ is a real form of the semisimple part of $d\rho(\g)$. If $\s$ is a split real form there is always a corresponding real form of $\pv$. On the other hand, in the case that $\s$ is compact, there need not be a corresponding real form of $\pv$. If there is a real structure $\ster$ on $\ste$ such that $\pvr$ is a real form of $\pv$, then $\rho(\Grea)$ has one-dimensional center and its semisimple part $S$ is compact. The restriction to $\ster$ of the relative invariant $P$ of $\pv$ is invariant along the orbits of $S$, which foliate an open subset of $\ster$. Since $S$ is compact it preserves a positive definite inner product and so these orbits lie in spheres. In general this is incompatible with the relative invariance of the restriction of $P$ to $\ster$ unless $P$ has degree $2$. For a concrete example, consider the prehomogeneous action $\rho$ of $GL(2, \com)$ on $\ste = S^{3}(\com^{2})$ induced by the standard representation. The relative invariant $P$ is given by \eqref{pvexample} (viewed as a complex polynomial). An element $z_{1}u^{3} + z_{2}u^{2}v + z_{3}uv^{2} + z_{4}v^{3}$ of $S^{3}(\com^{2})$ is identified with the vector $z = (z_{1}, z_{2}, z_{3}, z_{4})^{t}$. The real structure on $\com^{2}$ defined by $\tau z = \bar{z}$ induces the real structures on $S^{3}(\com^{2})$ and $GL(2, \com)$ given by complex conjugation of the coefficients and the corresponding (split) real form is that of example \ref{cubicsplit} in the list below. The complexification of the quaternions $\quat$ is isomorphic to the endomorphism algebra of $\com^{2}$ and so the group $GL(1, \quat)$ of invertible quaternions is a different real form of $GL(2, \com)$. An identification of $\quat$ with $\com^{2}$ induces an identification of the simple part of $GL(1, \quat)$ with a copy of $SU(2)$ embedded in $GL(2, \com)$. Via $\rho$, $GL(1, \quat)$ acts on $\ste$. Suppose $\ster$ were a real form of $\ste$ so that $(GL(1, \quat), \ster)$ were a real form of $(GL(2, \com), \ste)$. Then $P$ would be invariant under the induced action of $SU(2)$ on $\ster$. Since $SU(2)$ leaves invariant a quadratic form $E$ on $\ster$, if $x \in \ster$ is a generic point for the action then there is a constant such that $P(x) = cE^{2}(x)$. By homogeneity, $P - cE^{2}$ vanishes on an open subset of $\ster$, so is identically zero. This contradicts the irreducibility of $P$.

\subsection{}\label{pvexamplesection}
If $1 \leq n < \dim \ste = m$ then the tensor product of $\pv$ with the standard representation of $GL(n, \com)$ is prehomogeneous if and only if the tensor product of the contragredient triple $(G, \rho^{\ast}, \ste^{\ast})$ with the standard representation of $GL(m - n , \com)$ is prehomogeneous (see section $2$ of \cite{Sato-Kimura} and section $7.1$ of \cite{Kimura}). This is because the Grassmannian $Gr(n, \ste)$ of $n$-dimensional subspaces of $\ste$ is isomorphic to the Grassmannian $Gr(m-n, \ste^{\ast})$, and the conditions are respectively equivalent to $G$ having a Zariski open orbit in $Gr(n, \ste)$ or $Gr(m-n, \ste^{\ast})$. Two prehomogeneous vector spaces are \textit{castling} transforms of one another when there is a triple $\pv$ such that one is isomorphic to the tensor product of $\pv$ with the standard representation of $GL(n, \com)$ and the other is isomorphic to the tensor product of the contragredient triple $(G, \rho^{\ast}, \ste^{\ast})$ with the standard representation of $GL(m - n , \com)$. A prehomogeneous vector space is \textit{reduced} if its dimension is minimal among those prehomogeneous vector spaces in its castling class, and Proposition $2.12$ of \cite{Sato-Kimura} shows that each castling class contains a reduced triplet, unique up to equivalence of prehomogeneous vector spaces. Two prehomogeneous vector spaces $\pv$ and $(G^{\prime}, \rho^{\prime}, \ste^{\prime})$ are said to be \textit{equivalent} if there is a linear isomorphism $\tau:\ste \to \ste^{\prime}$ and a rational group isomorphism $\si:\rho(G) \to \rho^{\prime}(G^{\prime})$ such that $\tau \circ \rho(g) = (\si\rho(g))\circ \tau$ for all $g \in G$. The essential point here is that the images $\rho(G)$ and $\rho^{\prime}(G^{\prime})$ are locally the same. The basic example is that $(SL(2, \com)\times SL(2, \com), \rho\times \rho, \com^{2}\tensor \com^{2})$ and $(SO(4, \com), \rho, \com^{4})$ are equivalent, where in both cases $\rho$ indicates the standard representation, because $SL(2, \com)\times SL(2, \com)$ and $SO(4, \com)$ are locally isomorphic, although not isomorphic.

Sato and Kimura classified the reduced irreducible complex prehomogeneous vector spaces up to equivalence. There are $29$ that are regular, a unique one that is not regular though admitting a relative invariant, and several that admit no relative invariants (see section $7$ of \cite{Sato-Kimura} or the appendix to \cite{Kimura}). Interesting examples of solutions of \eqref{ma} are obtained via Theorem \ref{pvtheorem} from the real forms of the regular spaces appearing in the classification. The nonregular cases are not of interest here, because they do not yield solutions to \eqref{ma}. 

In order to summarize the Sato-Kimura classification conceptually, two special classes of prehomogeneous vector spaces are described now. A prehomogeneous vector space is said to be \textit{trivial} if it has the form $(G \times GL(m, \com), \rho \tensor \om, \com^{n} \tensor \com^{m})$, where $\om$ is the standard representation of $GL(m, \com)$, and $\rho$ is an $n$-dimensional representation of the semisimple algebraic group $G$. As is commented after Definition $3.5$ in \cite{Sato-Kimura}, a trivial prehomogeneous vector space is reduced unless $G = SL(n, \com)$ with $m > n > m/2$, in which case the castling transform is reduced. By Proposition $5.1$ of \cite{Sato-Kimura}, a trivial prehomogeneous vector space is regular if and only if $m = n$.

Let $G$ be a connected, simply-connected simple complex Lie group with $\integer$-graded Lie algebra $\g = \oplus_{k}\g_{k}$. The connected subgroup $G_{0}$ of $G$ with Lie algebra $\g_{0}$ is reductive and gradation preserving. By \cite{Vinberg-weylgroup}, the action of $G_{0}$ on $\g_{1}$ induced from the adjoint action of $G$ on $\g$ has only finitely many orbits, so is prehomogeneous. A prehomogeneous vector space equivalent to one of the form $(G_{0}, \ad_{G}, \g_{1})$ is said to be \textit{parabolic} and to have the \textit{type} which $G$ has in the classification of simple Lie algebras. This notion is due to H. Rubenthaler in \cite{Rubenthaler-espaces}, where the irreducible regular prehomogeneous vector spaces of parabolic type are classified.

Implicit in Rubenthaler's \cite{Rubenthaler-espaces} and \cite{Rubenthaler-nonparabolic} is the reformulation of the Sato-Kimura classification as stating that a reduced regular irreducible prehomogeneous complex vector space for a complex linear algebraic group $G$ is equivalent to one of one of the following types:
\begin{enumerate}
\item\label{rube1} A trivial prehomogeneous vector space $(G \times GL(n, \com), \rho \tensor \om, \com^{n} \tensor \com^{n})$, where $\rho$ is an $n$-dimensional irreducible representation of the semisimple algebraic group $G\neq SL(n, \com)$.
\item\label{rube2} A prehomogeneous vector space of parabolic type.
\item\label{rube3} The prehomogeneous vector space obtained by restricting a parabolic prehomogeneous vector space $(G_{0}, \ad_{G}, \g_{1})$ to some subgroup $H \subset G_{0}$.
\end{enumerate}
The $G = SL(n, \com)$ case is excluded from \eqref{rube1} because this case is parabolic of type $A_{2n+1}$, so is included in \eqref{rube2}. The spaces of type \eqref{rube3} are six in number, and arise from inclusions as subgroups of larger orthogonal groups of various Spin groups and the group $G_{2}$; see Table $1$ and Remark $8.2$ in \cite{Rubenthaler-nonparabolic}. This version of the classification is obtained \textit{a posteriori}, by comparing the Sato-Kimura classification with that of the parabolic prehomogeneous spaces. 

While listing all the possible real forms is straightforward in principle, since there are in general several real forms corresponding to each of the $29$ entries in the Sato-Kimura table, it did not seem useful to do so here. 
Instead an incomplete list is given, with the intention of illustrating the possibilities and highlighting some interesting examples. The explicit forms of the representation $\rho$ and the character $\chi$ are omitted in some cases, for which they should be clear, or for which their definitions would occupy too much space. The number in parentheses at the end of each entry in the list indicates the corresponding entry in table I of section $7$ of \cite{Sato-Kimura}. When the prehomogeneous vector space is a real form of a prehomogeneous vector space of parabolic type, this is indicated.

\begin{list}{\arabic{enumi}.}{\usecounter{enumi}}
\renewcommand{\theenumi}{Step \arabic{enumi}}
\renewcommand{\theenumi}{(\arabic{enumi})}
\renewcommand{\labelenumi}{\textbf{Level \theenumi}.-}
\item\label{matrixex} A real form of a regular trivial prehomogeneous vector space. $\Grea = G \times GL(n, \rea)$ where $G$ is a real form of a connected semisimple algebraic Lie group with irreducible action $\si$ on $\com^{n}$, $\ster = \rea^{n} \tensor (\rea^{n})^{\ast}$, 
$\rho(g_{1}, g_{2})X = \si(g_{1})Xg_{2}^{t}$, $\chi(g_{1}, g_{2}) = \det(\si(g_{1}))\det(g_{2})= \det(g_{2})$, and $P(X) = \det X$. When $G = SL(n, \rea)$ this is parabolic of type $A_{2n+1}$. $(1)$. 
\item\label{hermitianex} $\Grea = GL(1, \rea) \times SL(n, \com)$, $\ster = \{X \in \com^{n} \tensor (\com^{n})^{\ast}: \bar{X}^{t} = X\}$, 
$\rho(r, g)X = rgX\bar{g}^{t}$, $\chi(r, g) = r^{2}|\det g|^{2}$, and $P(X) = \det X$. Parabolic of type $A_{2n+1}$. $(1)$.
\item\label{symex} $\Grea = GL(n, \rea)$, $\ster = \{X \in \rea^{n}\tensor (\rea^{n})^{\ast}: X^{t} = X\}$, $\rho(g)X =gxg^{t}$, $\chi = \det^{2}(g)$, and $P(X) = \det X$. Parabolic of type $C_{n}$. $(2)$.
\item $\Grea = GL(2n, \rea)$, $\ster = \{X \in \rea^{2n}\tensor \rea^{2n\, \ast}: X^{t} = -X\}$, $\rho(g)X = gXg^{t}$, $\chi(g) = \det g$, and $P(X) = \pfaff X$ is the Pfaffian of $X$. Parabolic of type $D_{2n}$. $(3)$.
\item\label{qhermex} $\Grea = GL(1, \rea) \times SL(n, \quat)$, $n > 1$, $\ster$ comprises the $n \times n$ quaternionic Hermitian matrices, $\chi = |\det|^{2}$, and $P(X) = \det X$. Parabolic of type $D_{2n}$. $(3)$.
\item\label{cubicsplit} $\Grea = GL(2, \rea)$, $\ster = S^{3}(\rea^{2})$, $\chi = \det^{6}$, the action of $\Grea$ is that induced from the standard representation, $P$ is the discriminant \eqref{pvexample} of a cubic form. Parabolic of type $G_{2}$. $(4)$.
\item $\Grea = GL(m, \rea)$ and $\ster = \ext^{3}(\rea^{m})$, where $m = 6, 7, 8$. Parabolics of types $E_{6}$, $E_{7}$, and $E_{8}$. $(5)$, $(6)$, and $(7)$.
\item\label{233ex} $\Grea = SL(3, \rea) \times SL(3, \rea) \times GL(2, \rea)$ acting on $\rea^{3}\tensor \rea^{3}\tensor \rea^{2}$ via the outer tensor product of the standard representations, $\chi = \det^{12}$, $P$ is the hyperdeterminant of format $(2, 3, 3)$ described below. Parabolic of type $E_{6}$. $(12)$. 
\item $\Grea = Sp(n, \rea) \times GL(2m, \rea)$, $2 \leq 2m \leq n$,  $\ster = \rea^{2m}\tensor \rea^{2n\,\ast}$, $\rho(g_{1}, g_{2}) = g_{1}Xg_{2}^{t}$, $\chi = \det \tensor \det$, and $P(X) = \pfaff (X^{t}JX)$ where $J$ is an almost complex structure on $\rea^{2n}$ and $Sp(n, \rea) =\{A \in \eno(\rea^{2n}): A^{t}JA = J\}$. Parabolic of type $C_{n+m}$. $(13)$.
\item\label{soex} $\Grea = SO(p, n-p)\times GL(m, \rea)$, $n \geq 2m \geq 2$, $\ster = \rea^{m}\tensor \rea^{n\,\ast}$, $\rho(g_{1}, g_{2}) = g_{1}Xg_{2}^{t}$, $\chi = \det \tensor \det^{2}$, and $P(X) = \det X^{t}E_{p}X$ where $SO(p, n-p) = \{A \in \eno(\rea^{n}): A^{t}E_{p}A = E_{p}\}$. Parabolic of type $B_{(n-1)/2}$ or $D_{n/2}$. $(15)$.
\item\label{e6ex} $\Grea = GL(1, \rea) \times E_{6(-26)}$, $\ster$ is the $27$-dimensional real exceptional Jordan algebra of $3 \times 3$ Hermitian matrices over the octonions, with $GL(1, \rea)$ acting by scalar multiplication, and $P$ is the invariant cubic form on $\ster$ given by the determinant. 
Parabolic of type $E_{7}$. $(27)$.
\item\label{e6ex2} $\Grea = GL(1, \rea) \times E_{6(6)}$, $\ster$ is the $27$-dimensional space of $3 \times 3$ Hermitian matrices over the split octonions, with $GL(1, \rea)$ acting by scalar multiplication, and $P$ is the invariant cubic form on $\ster$ given by the determinant. Parabolic of type $E_{7}$. $(27)$.
\item\label{e7ex} $\Grea = GL(1, \rea) \times E_{7(7)}$, $\ster$ is the $56$-dimensional irreducible representation of the split real form $E_{7(7)}$, with $GL(1, \rea)$ acting by scalar multiplication, and $P$ is the invariant quartic form.  Parabolic of type $E_{8}$. $(29)$.
\end{list}
The polynomial associated with the $n = 3$ cases of \ref{hermitianex} and \ref{qhermex}, as well as \ref{e6ex}, has the form
\begin{equation}\label{cartpoly}
x_{1}x_{2}x_{3} - x_{1}|z_{2}|^{2} - x_{2}|z_{3}|^{2} - x_{3}|z_{1}|^{2} + \bar{z}_{1}(\bar{z}_{2}z_{3}) + (\bar{z}_{3}z_{2})z_{1},
\end{equation}
in which the $x_{i}$ are real and the $z_{i}$ are complex numbers, quaternions, or octonions, respectively. Examples \ref{matrixex} (with $G = SL(n, \rea)$) and \ref{hermitianex} are real forms of the complex prehomogeneous vector space $(SL(n, \com)\times GL(n, \com), \com^{n}\tensor \com^{n\,\ast}, \rho(g_{1}, g_{2})X = g_{1}Xg_{2}^{t})$. Their $n = 3$ cases yield respectively $x_{1}x_{5}x_{9} - x_{1}x_{6}x_{8} - x_{2}x_{4}x_{9} + x_{3}x_{4}x_{8} + x_{2}x_{6}x_{7} - x_{3}x_{5}x_{7}$
and the polynomial \eqref{cartpoly} (with $z_{1} = x_{4}+ \j x_{5}$, etc.), which solve \eqref{ma} with $m = 3$ and $\kc = -2$ and $\kc = 128$, respectively. Incidentally, this shows that these polynomials are not affinely equivalent, for were they, by \eqref{hgf} they would solve \eqref{ma} with $\kc$ of the same sign. 

\subsection{}\label{hyperdeterminantsection}
Interesting examples of affine spheres are obtained as the nonzero level sets of certain hyperdeterminants, e.g. \eqref{chdet}. While these examples arise as relative invariants of prehomogeneous vector spaces, they merit separate mention because they can be found by a different line of thought which is also suggestive. Let $\ste(1), \dots, \ste(r)$ be complex vector spaces of dimensions $k_{1} + 1, \dots, k_{r} + 1$. The \textit{hyperdeterminant of format} $(k_{1}+1, \dots, k_{r}+1)$ is by definition a polynomial defining the subvariety of $\proj((\ste(1)\tensor \dots \ste(r))^{\ast})$ dual to the image $\proj(\ste(1))\times \dots \times \proj(\ste(r))$ under the Segre embedding into $\proj(\ste(1)\tensor \dots \tensor \ste(r))$. The hyperdeterminant is determined up to sign by the requirement that its coefficients be integers and that it be irreducible over $\integer$. It is nontrivial exactly when the variety dual to the Segre variety is a hypersurface, which holds if and only if $k_{j} \leq \sum_{i \neq j}^{r}k_{i}$ for all $1 \leq j \leq r$ (see Corollary $5.10$ of chapter $1$ and Theorem $1.3$ of chapter $14$ of \cite{Gelfand-Kapranov-Zelevinsky}). Write $\hdet(A)$ for the hyperdeterminant of $A \in \ste(1)^{\ast}\tensor \dots \tensor \ste(r)^{\ast}$. 

\begin{corollary}\label{hyperdeterminantcorollary}
Let $k_{1}, \dots, k_{r}$ be positive integers satisfying $k_{j} \leq \sum_{i \neq j}^{r}k_{i}$ for $1 \leq j \leq r$, let $\ste(1), \dots, \ste(r)$ be complex vector spaces of dimensions $k_{1} + 1, \dots, k_{r} + 1$, and write $\stw = \ste(1)^{\ast} \tensor \dots \tensor \ste(r)^{\ast}$. Let $G = GL(\ste(1))\times \dots \times GL(\ste(r))$ and let $\rho$ be the irreducible representation of $G$ on $\stw$ arising as the contragredient of the outer product of the standard representations of the factors $GL(\ste(j))$. If $(G, \rho, \stw)$ is prehomogeneous, then it is regular with fundamental relative invariant given by the hyperdeterminant $\hdet$, and so the Hessian determinant of $\hdet$ is equal to a nonzero scalar multiple of a power of $\hdet$. Consequently there hold the following.
\begin{enumerate}
\item\label{hdp1} A nonzero level set of the hyperdeterminant $\hdet$ in the set $\stw_{\rea}$ of real points of any real form of $(G, \rho, \stw)$ is a proper homogeneous affine sphere. 
\item A nonzero level set in $\stw$ of the polynomial $|\hdet|^{2}$ is a proper homogeneous affine sphere.
\end{enumerate}
\end{corollary}

\begin{proof}
The assumption on the format guarantees that $\hdet$ is nontrivial. By construction, $\hdet$ is a relative invariant for the irreducible representation $\rho$ of $G$ on $\stw$. If $(G, \rho, \stw)$ is prehomogeneous, then, since it is irreducible, it admits up to scalars at most one irreducible relative invariant. Since the Hessian of a relative invariant is again a relative invariant, it will be the case that the Hessian determinant of $\hdet$ is a nonzero multiple of a power of $\hdet$ provided that the former is nonzero. This is implicit in the proof that the dual variety of the Segre embedding is a hypersurface; see section $5$ of chapter $1$ of \cite{Gelfand-Kapranov-Zelevinsky}, Theorems $5.1$ and $5.3$ in particular; alternatively, by Proposition $4.25$ of \cite{Sato-Kimura} it follows from the reductivity of the stabilizer in $G$ of a generic point of $\stw$. Thus when the format $(k_{1}+1)\times \dots \times (k_{r}+1)$ is such that the hyperdeterminant exists and $(G,\rho, \stw)$ is prehomogeneous, then $(G,\rho, \stw)$ is regular with fundamental relative invariant given by $\hdet$, and therefore $\hdet$ solves an equation of the form \eqref{ma}. The remaining claims follow as in the proof of Theorem \ref{pv2theorem}.
\end{proof}

It makes sense to speak of the product of the real general linear groups acting on the real form of $\ste(1)^{\ast}\tensor \dots \tensor \ste(r)^{\ast}$ under complex conjugation, and the restriction to this real form of the hyperdeterminant. For this real form of $(G, \rho, \stw)$, conclusion \eqref{hdp1} of Corollary \ref{hyperdeterminantcorollary} in conjunction with Theorem \ref{pvtheorem} yields that a connected component of a nonzero level set in $\ste(1)_{\rea}^{\ast}\tensor \dots \tensor \ste(r)_{\rea}^{\ast}$ of the hyperdeterminant, is a homogeneous proper affine sphere.

In general $(G, \rho, \stw)$ need not be prehomogeneous. An obvious necessary condition for there to be an open $G$ orbit in $\stw$ is that $\sum_{i = 1}^{r}(k_{i}+1)^{2} - \prod_{i = 1}^{r}(k_{i}+1) = \dim G - \dim \stw \geq r - 1$. The $r-1$ arises because the dilations in the $r$ factors all induce dilations on $\stw$. 
\begin{lemma}\label{reducedlemma}
Let $k_{1}, \dots, k_{r}$ be positive integers satisfying $k_{j} \leq \sum_{i \neq j}^{r}k_{i}$ for $1 \leq j \leq r$ and suppose that the representation $\rho$ of $G = GL(k_{1}+1, \com)\times \dots \times GL(k_{r}+1, \com)$ on the dual $\stw$ of the outer tensor product of the standard representations of its factors is prehomogeneous. Then the hyperdeterminant is the relative invariant obtained via castling transformations from one of the following polynomials.
\begin{enumerate}
\item The determinant of an $n\times n$ matrix viewed as the fundamental relative invariant of the irreducible reduced trivial prehomogeneous vector space $(SL(n, \com)\times GL(n, \com), \rho, \com^{n}\tensor \com^{n})$.
\item The Cayley hyperdeterminant of format $(2, 2, 2)$ viewed as the fundamental relative invariant of the irreducible reduced prehomogeneous vector space $(SL(2, \com)\times SL(2, \com)\times GL(2,\com), \rho, \com^{2}\times \com^{2}\times \com^{2})$.
\item The hyperdeterminant of format $(3, 3, 2)$ viewed as the fundamental relative invariant of the irreducible reduced prehomogeneous vector space $(SL(3, \com)\times SL(3, \com)\times GL(2,\com), \rho, \com^{2}\times \com^{2}\times \com^{2})$.
\end{enumerate}
\end{lemma}

\begin{proof}
By Theorem $2$ of \cite{Kac-nilpotent} or Theorem $5$ of \cite{Parfenov}, the action of $G$ on $\stw$ has finitely many orbits exactly when the format $(k_{1} + 1, \dots, k_{r} + 1)$ is one of $(a)$, $(a, b)$, $(2, 2, c)$, or $(2, 3, c)$ where $c \geq 3$. That $G$ have finitely many orbits in $\stw$ implies that $(G, \rho, \stw)$ is prehomogeneous. For hyperdeterminants it is interesting only to consider cases with $r \geq 3$. In this case the only formats for which there exists a nontrivial hyperdeterminant and $G$ has finitely many orbits are $(2, 2, 2)$, $(2, 2, 3)$, $(2, 3, 3)$, and $(2, 3, 4)$, and the hypotheses of Corollary \ref{hyperdeterminantcorollary} are satisfied for these hyperdeterminants.

The format $(2, 2, 2)$ gives the Cayley hyperdeterminant \eqref{chdet}. Since $SL(2, \com)\times SL(2, \com)$ is locally isomorphic to $SO(4, \com)$, this hyperdeterminant can also be seen as arising as the relative invariant associated with the action of $SO(4,\com)\times GL(2, \com)$ on $\com^{4}\tensor \com^{2}$; this is a special case of item $(15)$ in the Sato-Kimura list (see \ref{soex} in the list above). Likewise the hyperdeterminant of format $(2, 3, 3)$ corresponds directly to entry $(12)$ in the Sato-Kimura list; see \ref{233ex} in the list above.

If $ab > c$ then the tensor product of format $(a, b, c)$ is castling equivalent to the tensor product of the format $(a, b, c - ab)$. In particular the formats $(2, 3, 4)$, $(2, 2, 3)$, and $(1, 2, 2)$ are identified modulo castling and equivalence. The prehomogeneous vector space $(GL(1, \com)\tensor GL(2, \com)\tensor GL(2, \com), \rho, \stw = \com \tensor \com^{2} \tensor \com^{2})$ is equivalent to $(SL(2, \com)\times GL(2, \com), \rho, \com^{2} \tensor \com^{2})$, which in turn is equivalent to $(GL(1, \com)\tensor SO(4, \com), \om, \com^{4})$, where $\om$ is the standard representation of $SO(4, \com)$ and $GL(1,\com)$ acts by scalar multiplication. By Proposition $5.18$ of \cite{Sato-Kimura} there is a one-to-one correspondence between relative invariants of a prehomogeneous vector space and relative invariants of its castling transform under which irreducible relative invariants correspond to irreducible relative invariants. It follows that the hyperdeterminants of formats $(2, 3, 4)$ and $(2, 2, 3)$ are obtained from each other via castling, and both are obtained via castling from the fundamental invariant of $(SL(2, \com)\times GL(2, \com), \rho, \com^{2} \tensor \com^{2})$, which is simply the determinant of $2\times 2$ matrices.

The correspondence between relative invariants of castling equivalent prehomogeneous vector spaces is given explicitly on pp.$67-68$ of \cite{Sato-Kimura}. If $(G\times GL(n, \com), \rho \tensor \om, \ste(m)\tensor \ste(n))$ and $(G\times GL(m-n, \com), \rho^{\ast}\tensor \om_{n}, \ste(m)^{\ast}\tensor \ste(m-n))$ are castling equivalent, then a relative invariant $P$ of the first is given as a polynomial in the $\binom{m}{n}$ minors of an element of $\ste(m)\tensor \ste(n)$, and the corresponding relative invariant of the second prehomogeneous vector space is given by the same polynomial in the $\binom{m}{m-n}$ minors of $\ste(m)^{\ast}\tensor \ste(m-n)$. For example, the relative invariant of $GL(1, \com)\times GL(2, \com) \times GL(2, \com)$ acting on $\com^{2}\tensor \com^{2}$ is given by the determinant $P(x_{11}, x_{12}, x_{21}, x_{22}) = x_{11}x_{22} - x_{12}x_{21}$. For the castling equivalent space with $GL(2, \com)\times GL(2, \com)\times GL(3, \com)$ acting on $\com^{2}\tensor \com^{2}\tensor \com^{3}$, take as coordinates the array $x_{abi}$ where $a, b \in \{1, 2\}$ and $i \in \{1, 2, 3\}$, and regard these as the entries of a $4 \times 3$ matrix, in which the expression $2a + b -2$ in terms of the first two indices labels the row and the third index labels the column. For $1 \leq \al < \be < \ga \leq 4$ let $X_{\al\be\ga}$ be the $3 \times 3$ minor of this $4\times 3$ matrix corresponding to the rows $\al$, $\be$, and $\ga$. Then the relative invariant $Q$ corresponding to $P$ via castling is $Q(x_{abi})  = X_{123}X_{234} - X_{124}X_{134}$.
Modulo notation, this expression coincides with that given in Example $3.8$ of chapter $14$ of \cite{Gelfand-Kapranov-Zelevinsky} for the hyperdeterminant of format $(2, 2, 3)$. 

The prehomogeneous vector space $(G, \rho, \stw)$ is equivalent to $(\bar{G}, \rho, \stw)$ where $\bar{G} = SL(k_{1}+1, \com)\times \dots \times SL(k_{r-1}+1, \com) \times GL(k_{r}, \com)$, because the images $\rho(G)$ and $\rho(\bar{G})$ are the same. Note that $(\bar{G}, \rho, \stw)$ is a trivial prehomogeneous vector space if and only if $\prod_{i = 1}^{r-1}(k_{i} + 1)\leq k_{r} + 1$. Since $(\bar{G}, \rho, \stw)$ is an irreducible prehomogeneous vector space, it follows from Proposition $7.47$ of \cite{Kimura} that it is reduced if and only if either it is a trivial prehomogeneous vector space, or $r \leq 3$ and, when $r = 3$, $2 \leq k_{1} = k_{2} \leq 3$. In the trivial case, for $(SL(k_{1} +1, \com)\times GL(k_{2} + 1, \com), \rho, \stw)$ to be reduced and regular it must be that $k_{2} = k_{1}$; in this case the relative invariant is simply the determinant. Suppose $(\bar{G}, \rho, \stw)$ is not trivial and is reduced. Then the format is either $(2, 2, 2)$ or $(3, 3, 2)$. Together with Corollary \ref{hyperdeterminantcorollary} this completes the proof. 
\end{proof}

It would be interesting to explain what it means geometrically that affine spheres arise as the level sets of a polynomial that is not reduced, and what is the geometric relation between affine spheres arising as the level sets of polynomials related via castling.

None of the preceding precludes the possibility that the Hessian determinant of a nontrivial hyperdeterminant, such as that of format $(2, 2, 2, 2)$, is a multiple of a power of the hyperdeterminant, although the underlying $G$ action is not prehomogeneous. 

\subsection{}\label{convexconesection}
Some of the examples of affine spheres obtained from relative invariants of prehomogeneous vector spaces were already obtainable from other constructions. The affine spheres given by Theorem \ref{pvtheorem} are necessarily homogeneous. All the locally uniformly convex examples arising this way were obtainable by a construction of Calabi. The \textit{characteristic function} $\phi_{\Om}$ of a proper open convex cone $\Om \subset\rea^{n+1}$ is the positive homogeneity $-n-1$ real analytic function $\phi_{\Om}(x)  = \int_{\Om^{\ast}}e^{-x^{p}y_{p}}dy$, where $\Om^{\ast} =\{y \in \rea^{n+1\,\ast}: x^{p}y_{p} > 0 \,\,\text{for all}\,\, x \in \bar{\Om}\setminus\{0\}\}$.
\begin{theorem}[E. Calabi, \cite{Calabi-completeaffine}; T. Sasaki, \cite{Sasaki}]\label{calabitheorem}
If a closed subgroup $G \subset GL(n+1, \rea)$ acts transitively as automorphisms of a nonempty proper open convex cone $\Om$ then the orbit of any point of $\Om$ under the action of the subgroup $SG = G \cap SL(n+1, \rea)$ is a hyperbolic affine sphere, homogeneous under the action of $SG$, and contained in a level set of the characteristic function of $\Om$.
\end{theorem}
\begin{proof}
Since $g\cdot \phi_{\Om} = \det \ell(g)\phi_{ \Om}$ for $g \in \Aut(\Om)$, by \eqref{hgf}, $\phi_{\Om}^{-2}\H(\log \phi_{\Om})$ is $\Aut(\Om)$ invariant, and so, because $\Om$ is homogeneous, must equal a constant $c \in \rea$; because $\hess(\log \phi_{\Om})$ is positive definite, $c > 0$. Since $\phi_{\Om}$ is homogeneous, this shows there is a $c > 0$ such that $\H(\log \phi_{\Om}) = c\phi_{\Om}^{2}$. By Theorem \ref{ahtheorem} the level sets of $\phi_{\Om}$ are the proper affine spheres (of Theorem \ref{cytheorem}) foliating the interior of $\Om$. From the relative invariance of $\phi_{\Om}$ it follows straightforwardly that the group $\Aut(\Om)\cap SL(n+1, \rea)$ acts transitively on the level sets of $\phi_{\Om}$, so these are homogeneous proper affine spheres. 
\end{proof}

For example, the affine spheres determined by examples \ref{hermitianex}, \ref{symex}, and \ref{qhermex} in the list above, were found in \cite{Calabi-completeaffine} by applying Theorem \ref{calabitheorem}. These three examples, and that of \ref{e6ex}, corresponding to the octonionic Hermitian matrices, are exactly the finite-dimensional simple formally real Jordan algebras (see Theorem V.$3.7$ of \cite{Faraut-Koranyi}). As is shown by Theorems \ref{ahtheorem} and \ref{pvtheorem}, the convexity plays no essential role. By Theorem $2.3$ of \cite{Faraut-Gindikin} there is a bijective correspondence associating with a semisimple real (necessarily unital) Jordan algebra the symmetric cone comprising the connected component containing the unit of the set of its invertible elements. This cone is the interior of the cone of squares, and is acted on transitively by the identity component of the structure group of the Jordan algebra, and so its complexification is a prehomogeneous vector space, which is irreducible if the Jordan algebra is simple. These are the regular irreducible prehomogeneous vector spaces of commutative parabolic type. The resulting Jordan algebra is either a real form of a simple complex Jordan algebra, or the underlying real Jordan algebra of a simple complex Jordan algebra. Theorem $5.12$ of R. Hildebrand's \cite{Hildebrand-parallelcubicform} shows that the homogeneous affine spheres obtained from simple real unital Jordan algebras via this observation and Theorem \ref{pvtheorem} are exactly the affine spheres indecomposable with respect to the Calabi product and having cubic form parallel with respect to the Levi-Civita connection of the equiaffine metric. The simplest example given by Theorem \ref{pvtheorem} that does not arise from Jordan algebras is \eqref{pvexample}. 

To conclude, it is explained how to give sense to the statement that the complexification of a homogeneous proper open convex cone $\Om$ is a regular prehomogeneous vector space, and, using this, to obtain the result of Vinberg that the square of its characteristic function is a rational function. This yields an abundance of rational solutions to an equation of the form \eqref{mai}. Since $\Om$ is connected the connected component $\Aut_{0}(\Om)$ of the identity in $\Aut(\Om)$ also acts transitively on $\Om$. Since $\Om$ is a cone, Propositions I.$13$ and I.$14$ of \cite{Vinberg}, show that the normalizer in $\Aff(n+1, \rea)$ of $\Aut_{0}(\Om)$ is a linear algebraic group and its connected component of the identity coincides with $\Aut_{0}(\Om)$. Hence $\Aut(\Om)$ is a finite index subgroup of some linear algebraic group $G \subset GL(\rea^{n+1})$. The orbits of $G$ are finitely many disjoint images of $\Om$. Hence the complexification $G^{\com}$ of $G$ acts on $\com^{n+1}$ prehomogeneously, in such a manner that the real points of the open orbit are the orbits of $G$. 
The action of $G$ need not preserve $\Om$. Rather, $G\Om$ is a finite union of disjoint linear images of $\Om$ by elements of $G$. Let $G_{0}, \dots, G_{r}$ be the connected components of $G$, where $G_{0}$ is the component of the identity. For $1 \leq i \leq r$ there is $g_{i} \in G_{i}$ such that $G_{i} = g_{i}G_{0}$. Then $\Om_{i} = G_{i}\Om = g_{i}G_{0}\Om = g_{i}\Om$ is a linear copy of $\Om$. The different $\Om_{i}$ need not be disjoint, but there are at most finitely many of them, $\Om_{0} = \Om, \Om_{1}, \dots, \Om_{s}$. It follows from the definition that for any $g \in \aff(n+1, \rea)$ fixing the vertex of $\Om$ there holds $(g \cdot \phi_{\Om})(x) = \det \ell(g)\phi_{g\Om}(x)$. Since $G$ contains positive dilations, $e^{t}g_{i}\Om = \Om_{i}$, and so, by replacing $g_{i}$ by $e^{t}g_{i}$, it can be assumed that $\det \ell(g_{i}) = \pm 1$. Then $\phi_{\Om_{i}}(g_{i}x) =  \phi_{g_{i}\Om}(g_{i}x) = \pm(g_{i}\cdot \phi_{\Om})(x)$. It follows that $\phi_{\Om}^{2}$ extends coherently to the union of the images $\Om_{i}$. Since this extended function $\phi_{\Om}^{2}$ is real analytic it can be extended to a complex analytic function defined on the open orbit of $G^{\com}$, and transforms in the manner necessary for it to be a relative invariant for the action of $G^{\com}$. It needs to be shown that $\phi_{\Om}^{2}$ is rational. By Proposition $2.11$ of \cite{Kimura}, the group of characters of relative invariants of a complex prehomogeneous vector space $\pv$ is equal to the group of characters of $G$ equal to $1$ on the isotropy group of a point of $\ste$ in the open orbit of $G$. In the present setting, since the isotropy subgroup $H$ in $G$ of $x_{0} \in \Om$ coincides with the isotropy subgroup in $\Aut(\Om)$, and this group $H$ preserves the positive definite bilinear form given by the Hessian of $\log \phi_{\Om}$ at $x_{0}$, this group $H$ is contained in the kernel of the character $\det^{2} \ell(g)$ of $G$, and so its complexification $H^{\com}$, which is the stabilizer in $G^{\com}$ of $x_{0}$, is likewise contained in the kernel of $\det^{2}\ell(g)$. It follows that there is a relative invariant of $(G^{\com}, \com^{n+1})$ corresponding to this character. By homogeneity some nonzero multiple of $\phi_{\Om}^{2}$ coincides with this relative invariant on the open orbit, showing that $\phi_{\Om}^{2}$ is rational. Since $\phi_{\Om}$ has nondegenerate Hessian on $\Om$, the same is true for its extension to the complexification, and so this shows that $(G^{\com}, \com^{n+1})$ is a regular prehomogeneous vector space. In particular, the extension of $\phi_{\Om}^{2}$ is expressible as a constant multiple of a product of (possibly negative) integer powers of irreducible polynomials defining the singular set of the action of $G^{\com}$. This proves the following theorem due to Vinberg.
\begin{theorem}[E.~B. Vinberg; section III.$4$ of \cite{Vinberg}]\label{vinbergtheorem}
The square of the characteristic function of a homogeneous proper open convex cone equals the restriction to the cone of a rational function.
\end{theorem}

Since $\phi_{\Om}^{-2}$ has positive homogeneity $2(n+1)$ and is rational, it is possible that it is a polynomial, though it need not be so. For example, it follows from equation $(1)$ of section II.$2$ of \cite{Vinberg-automorphisms} or the explicit formula $(33)$ for $\phi_{\Om}$ in section III.$4$ of \cite{Vinberg} that for a self-dual homogeneous irreducible proper open convex cone $\Om$ the function $\phi_{\Om}^{-2}$ is a polynomial. More precisely, in this case $\phi_{\Om}^{-2}$ is a multiple of the determinant of the quadratic representation of the Euclidean Jordan algebra having $\Om$ as the interior of its cone of squares (see Propositions III.$4.2$ and III.$4.3$ of \cite{Faraut-Koranyi}). On the other hand, in \cite{Vinberg}, Vinberg constructed a $5$-dimensional homogeneous irreducible convex cone not linearly isomorphic to its dual, and the characteristic functions of this cone and its dual can be computed explicitly and their squared reciprocals are rational but not polynomial. 

\def\cprime{$'$} \def\cprime{$'$} \def\cprime{$'$} \def\cprime{$'$}
  \def\cprime{$'$} \def\cprime{$'$}
  \def\polhk#1{\setbox0=\hbox{#1}{\ooalign{\hidewidth
  \lower1.5ex\hbox{`}\hidewidth\crcr\unhbox0}}} \def\cprime{$'$}
  \def\Dbar{\leavevmode\lower.6ex\hbox to 0pt{\hskip-.23ex \accent"16\hss}D}
  \def\cprime{$'$} \def\cprime{$'$} \def\cprime{$'$} \def\cprime{$'$}
  \def\cprime{$'$} \def\cprime{$'$} \def\cprime{$'$} \def\cprime{$'$}
  \def\cprime{$'$} \def\cprime{$'$} \def\cprime{$'$} \def\cprime{$'$}
  \def\dbar{\leavevmode\hbox to 0pt{\hskip.2ex \accent"16\hss}d}
  \def\cprime{$'$} \def\cprime{$'$} \def\cprime{$'$} \def\cprime{$'$}
  \def\cprime{$'$} \def\cprime{$'$} \def\cprime{$'$} \def\cprime{$'$}
  \def\cprime{$'$} \def\cprime{$'$} \def\cprime{$'$} \def\cprime{$'$}
  \def\cprime{$'$} \def\cprime{$'$} \def\cprime{$'$} \def\cprime{$'$}
  \def\cprime{$'$} \def\cprime{$'$} \def\cprime{$'$} \def\cprime{$'$}
  \def\cprime{$'$} \def\cprime{$'$} \def\cprime{$'$} \def\cprime{$'$}
  \def\cprime{$'$} \def\cprime{$'$} \def\cprime{$'$} \def\cprime{$'$}
  \def\cprime{$'$} \def\cprime{$'$} \def\cprime{$'$} \def\cprime{$'$}
  \def\cprime{$'$} \def\cprime{$'$} \def\cprime{$'$} \def\cprime{$'$}
\providecommand{\bysame}{\leavevmode\hbox to3em{\hrulefill}\thinspace}
\providecommand{\MR}{\relax\ifhmode\unskip\space\fi MR }
\providecommand{\MRhref}[2]{%
  \href{http://www.ams.org/mathscinet-getitem?mr=#1}{#2}
}
\providecommand{\href}[2]{#2}

\end{document}